\documentclass[11pt]{amsart}

\usepackage[T1]{fontenc}
\usepackage{mathptmx,microtype,hyperref,graphicx,dsfont,booktabs}
\usepackage{amsthm,amsmath,amssymb}
\usepackage[foot]{amsaddr}
\usepackage{enumitem}
\usepackage[nameinlink,capitalise]{cleveref}
\usepackage[font=footnotesize,margin=2.5cm]{caption}
\usepackage[numbers,sort&compress]{natbib}
\usepackage[tableposition=top]{caption}

\crefname{theorem}{Theorem}{Theorems}
\crefname{thm}{Theorem}{Theorems}
\crefname{mainthm}{Theorem}{Theorems}
\crefname{lemma}{Lemma}{Lemmas}
\crefname{lem}{Lemma}{Lemmas}
\crefname{remark}{Remark}{Remarks}
\crefname{claim}{Claim}{Claims}
\crefname{subclaim}{Sub-claim}{Sub-claims}
\crefname{prop}{Proposition}{Propositions}
\crefname{proposition}{Proposition}{Propositions}
\crefname{defn}{Definition}{Definitions}
\crefname{corollary}{Corollary}{Corollaries}
\crefname{conjecture}{Conjecture}{Conjectures}
\crefname{question}{Question}{Questions}
\crefname{chapter}{Chapter}{Chapters}
\crefname{section}{Section}{Sections}
\crefname{figure}{Figure}{Figures}
\crefname{table}{Table}{Tables}
\newcommand{\crefrangeconjunction}{--}

\theoremstyle{plain}

\newtheorem{thm}{Theorem}[section]
\newtheorem*{thm*}{Theorem}

\newtheorem{lem}[thm]{Lemma}

\theoremstyle{definition}

\theoremstyle{remark}

\numberwithin{equation}{section}

\newcommand{\eps}{\varepsilon}
\renewcommand{\P}{{\bf P}}
\newcommand{\E}{{\bf E}}

\newcommand{\R}{{\mathbb R}}

\newcommand{\var}{\mathrm{Var}}

\newcommand{\A}{{\tt A}}
\newcommand{\B}{{\tt B}}

\author[P. Diaconis]{Persi Diaconis}

\address{Departments of Mathematics and Statistics, Stanford University}

\author[B. Kolesnik]{Brett Kolesnik}

\address{Department of Statistics, University of California, Berkeley}

\email{bkolesnik@berkeley.edu, diaconis@math.stanford.edu}

\begin{document}

\title[Randomized sequential importance sampling]
{
Randomized sequential importance sampling
for estimating the number of perfect
matchings in bipartite graphs
}

\dedicatory{In memory of Ron Graham}

\maketitle

\vspace{-0.15cm}
\begin{abstract}
We introduce and study randomized sequential importance sampling algorithms 
for estimating the number of perfect matchings in bipartite graphs. 
In analyzing their performance, we establish various  
non-standard central limit theorems. 
We expect 
our methods to be 
useful for  
other applied problems. 
\end{abstract}


\section{Introduction}\label{S_intro}

Sequential importance sampling is a widely used technique for 
Monte Carlo evaluation of intractable counting and statistical problems. 
Using randomized algorithms, we apply this method to 
estimate 
the number of perfect matchings in bipartite 
graphs for a suite of test problems \cite{DGH99} where  
careful analytics can be carried out. 
We think that our techniques should be useful 
for estimating the number of perfect matchings in other types of bipartite graphs, 
and also for a variety of applied problems, 
such as counting and testing for contingency tables \cite{CDHL05}
and for graphs with given degree sequence \cite{BD10}. 

Importance sampling 
uses a relatively simple measure $\mu$  
to obtain information about a more complicated measure
$\nu\ll\mu$ of interest. The Kullback--Leibler divergence 
$L=\E_\nu\log (d\nu/d\mu)$ relates the two. 
Recent work \cite{CD18} (see \cref{S_IS}) shows that, if $\log (d\nu/d\mu)$ is concentrated, roughly $e^L$ samples 
from $\mu$ 
are necessary and sufficient to 
well approximate 
quantities of the form $\int f d\nu$. 

In this work, $\mu$ is the uniform measure on the set of perfect matchings of a given graph,
and $\nu$ is some other measure on this set induced by a randomized sampling algorithm. 
This work is the first
to prove limit theorems for $\log (d\nu/d\mu)$ for a variety of graphs and algorithms, 
and thereby accurately gauge the efficiency of importance sampling in this context. 
Even for simple graphs and natural sampling algorithms, 
this turns out to be quite a non-standard problem. 

Typically 
$L\sim cn$ is asymptotically linear, 
resulting in algorithms with exponential running times.
However, the constants
$c$ involved are small enough
so that, for $n$ of practical interest, 
accurate estimates
of huge numbers  are available using reasonably small sample sizes.
Indeed, our algorithms compare well  
with polynomial Markov chain 
Monte Carlo algorithms  
\cite{MSV04,DJM,BSVV08}
within this range of $n$. 

\subsection{Setup}

Suppose that $B_n$ is a bipartite graph on vertex sets
$U_n=\{u_1,\ldots,u_n\}$ and $V_n=\{v_1,\ldots,v_n\}$ with various edges between them. 
Let ${\mathcal M}_n$ be the set of perfect matchings in $B_n$, supposing here and throughout that 
${\mathcal M}_n\neq\emptyset$. 
We identify a perfect matching in ${\mathcal M}_n$ with the  
permutation $\pi\in S_n$ such that $\pi(i)=j$ if $u_i$ and $v_j$ are matched
(and so speak of perfect matchings and permutations 
interchangeably). 

In a variety of statistical problems arising 
with censored or truncated data
it is important to understand the distribution of various statistics, e.g.,  
the number of involutions, 
cycles or fixed points of uniformly distributed elements of ${\mathcal M}_n$. 
For example, 
Efron and Petrosian \cite{EP92} 
need random matchings in a bipartite graph arising in 
an astrophysics problem; see \cite{DGH99} for discussion on how many such matchings
are required in tests for parapsychology.  
These are provably $\#{\rm P}$-complete
problems, so approximation is all that can be hoped for. 
See \cref{S_MT} for further discussion on matching theory. 

\begin{figure}[h!]
\centering
\includegraphics[scale=1]{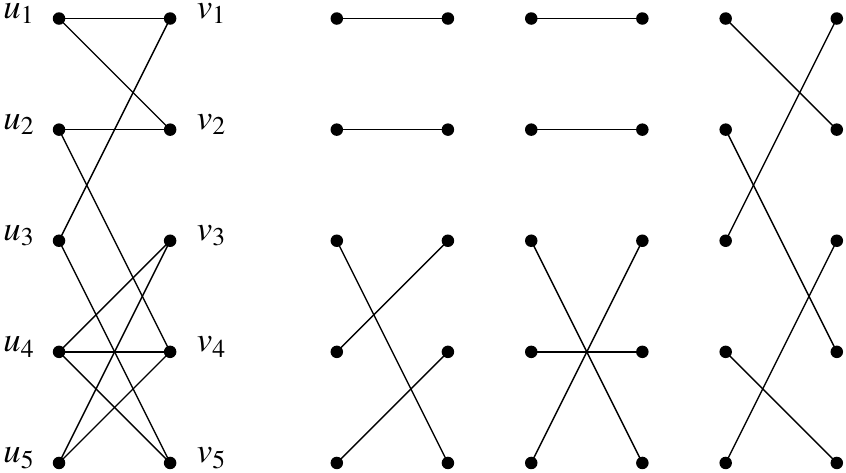}
\caption{A bipartite graph and its three perfect matchings.
}
\label{F_match}
\end{figure}

Importance sampling, selects $\pi\in {\mathcal M}_n$ 
according to a distribution $P(\pi)$ which is 
relatively easy to sample
from. For instance, consider 
the case that vertices in $U_n$ are matched in the fixed   
order $u_1,u_2,\ldots,u_n$, where in the $i$th step $u_i$ is matched with a 
vertex $v\in V_n$ uniformly at random amongst the remaining allowable options.  
Then 
$
P(\pi) = \prod_i |I_i|^{-1}$, 
where $I_i(\pi)\subset V_n$ is the set of vertices $v\in V_n$ which 
(a) are
not matched with any of  $u_1,u_2,\ldots,u_{i-1}$ and 
(b) such that if next $u_i$ is matched with 
$v$ then it would still be
possible to complete what remains to a perfect matching. 

For example, the graph in \cref{F_match} has three 
perfect matchings 
$\pi_1=12534$, $\pi_2=12543$ and $\pi_3=24153$. 
Under the algorithm described in the previous paragraph, 
$P(\pi_1)=1/4$, $P(\pi_2)=1/4$ and $P(\pi_3)=1/2$. 
To see this, note that $u_1$ can be matched with either $v_1$ or $v_2$. 
In the former case, $u_2$ must be matched with 
$v_2$, and then $u_3$ must be matched with $v_5$, 
and then finally $u_4$ can then be matched with either $v_3$
or $v_4$ (and then $u_5$ is matched with whichever of $v_3$ and $v_4$ remains
to complete the matching). 
In the latter case, once $u_1$ is matched with $v_2$, the rest of the matching is 
forced, as then there is only one way to complete to a perfect matching. 

The first key observation is that, letting $M_n=|{\mathcal M}_n|$ and $T(\pi)=P(\pi)^{-1}$, 
we have
\[
\E T(\pi)=\sum_{\pi} T(\pi)P(\pi) = M_n, 
\]
so $T(\pi)$ is an unbiased estimator of $M_n$. Moreover, 
if $\pi_1,\ldots,\pi_N$ is a sample from $P(\pi)$ then 
\[
\P_u(Q(\pi)\le\gamma_2)
\doteq \frac{1}{N}\sum_{i=1}^N \delta(Q(\pi_i)\le\gamma_2) T(\pi_i), 
\]
where on the left, $Q(\pi)$ is a statistic of interest
and $u=1/M_n$ is the uniform distribution on ${\mathcal M}_n$. 

As already mentioned, recent work \cite{CD18}
suggests that  
in many cases  
a sample of size 
$N\doteq e^L$ is necessary and sufficient 
to well approximate $M_n$ by an importance sampling
algorithm, where 
$L=\E_u\log(T(\pi)/M_n)$ is the Kullback--Leibler divergence
of the importance sampling measure $P(\pi)$ from the uniform
measure $u=1/M_n$ on ${\mathcal M}_n$. 
To apply this result, one needs to verify that  
$\log T(\pi)$ under $u$ is concentrated about its mean. 
See \cref{S_IS} for more details.

\subsection{Results}

Our main focus is the the test case of  
{\it Fibonacci matchings} 
\[
{\mathcal F}_{n,1}=\{\pi\in S_n:|\pi(i)-i|\le 1\}. 
\] 
Our techniques also extend to 
other related classes of graphs introduced in \cite{DGH99}, such as
{\it $t$-Fibonacci matchings}
\[
{\mathcal F}_{n,t}=\{\pi\in S_n:-1\le\pi(i)-i\le t\},  
\] 
and {\it distance-$d$ matchings}
\[
{\mathcal D}_{n,d}=\{\pi\in S_n:|\pi(i)-i|\le d\}. 
\] 

\begin{figure}[h!]
\centering
\includegraphics[scale=1]{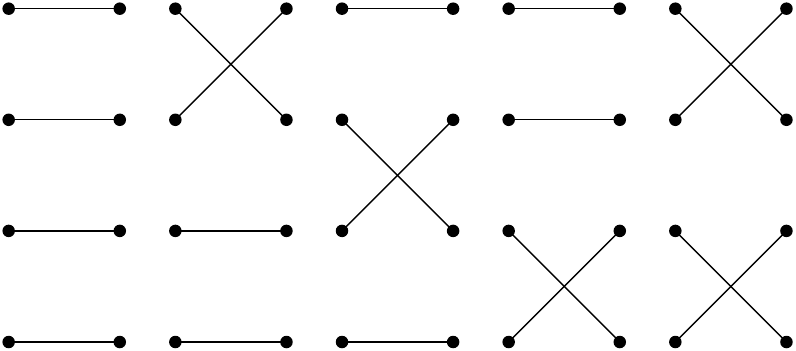}
\caption{The 5 Fibonacci matchings 
of size $n=4$. 
}\label{F_fib_n4}
\end{figure}

For simplicity, 
we restrict ourselves for the time being to the Fibonacci case. 
The reason for the name is that $F_{n,1}=|{\mathcal F}_{n,1}|$
is equal to the Fibonacci number $F_{n+1}$, 
as is easily seen by considering whether $\pi(1)=1$ or $\pi(1)=2$ 
(in which case necessarily $\pi(2)=1$). 
Although the size of ${\mathcal F}_{n,1}$ is known, 
our aim is to estimate 
$F_{n,1}$ by various importance sampling algorithms, so as  
to obtain a benchmark for its performance
and to gain insight towards applying these
methods in more complicated situations, in practice and  theory.

For example, when $n=4$, there are 5 Fibonacci matchings. 
Two distributions $P_1$ and $P_2$ are listed below, corresponding 
to proceeding in orders $1,2,3,4$ and $2,3,1,4$, in each step matching $i$
at random amongst its neighbors whose matchings
are yet to be determined.  
For instance $P_1(1234)=1/8$, since 1 can be matched with 1 or 2;  
then given $\pi(1)=1$, 2 can be matched with 2 or 3; then
given $\pi(12)=12$, 
 3 can be matched with 3 or 4, and then $\pi(4)=4$ is forced. 
Similarly $P_2(1324)=1/3$, since $2$ can be matched with 1, 2 or 3; then  
given $\pi(2)=3$,  the rest of $\pi$
is forced.

\begin{table}[h!]
\begin{center}
\caption{Two schemes for 
Fibonacci matchings  
${\mathcal F}_{4,1}$.}
\begin{tabular}{l|lllll}
\toprule
$\pi$&	1234&	2134&	1324&	1243&	2143\\\hline
$P_1(\pi)$&	$1/8$&	$1/4$&	$1/4$&	$1/8$& $1/4$ \\
$P_2(\pi)$&	$1/6$&	$1/6$&	$1/3$&	$1/6$& $1/6$\\
\bottomrule
\end{tabular}
\end{center}
\end{table}

A natural question to ask is how accurately can one estimate $F_{n,1}$
by an importance sampling algorithm $\A$ which in each step 
matches the current index $i$ with another index
uniformly at random (amongst the remaining 
allowable options). We consider three such algorithms 
$\A_r$, $\A_f$ and $\A_g$ which match indices 
$[n]$ in uniformly random order, 
the fixed
order $1,2,\ldots,n$,
and according to 
a certain greedy order. In each step of $\A_g$  
the smallest unmatched 
index $i$ is matched amongst those indices $i$ 
with the maximal number of remaining choices for $\pi(i)$. 
For instance, if in the first step of the algorithm 2 is 
matched with 1 or 2 (in which case $\pi(12)$ is determined), 
then 4 is matched in the next step; whereas if 2 is matched with 
3 (in which case $\pi(123)$ is determined), 
then 5 is matched in the next step. 

To analyze the performance of these 
algorithms, we prove the following central limit theorems.

\begin{thm}\label{T_main}
Consider the distributions $P_r(\pi)$, $P_f(\pi)$ 
and $P_g(\pi)$
on Fibonacci matchings 
$\pi\in {\mathcal F}_{n,1}$ obtained by algorithms 
$\A_r$, $\A_f$ and $\A_g$,  which 
 in random, fixed and greedy orders (as above), 
sequentially match 
indices randomly amongst the 
remaining allowable options.  
Then, under 
the uniform measure $u=1/F_{n,1}$, 
\[
\frac{\log T_r(\pi)-\mu_r n}{\sigma_r\sqrt{n}},\quad 
\frac{\log T_f(\pi)-\mu_f n}{\sigma_f\sqrt{n}},\quad
\frac{\log T_g(\pi)-\mu_g n}{\sigma_g\sqrt{n}} 
\]
all converge in distribution to standard normals 
$N(0,1)$, 
where
\[
\mu_r 
\doteq 0.4944,\quad 
\mu_f
\doteq0.5016,\quad 
\mu_g
\doteq 0.4913,\]
\[
\sigma^2_r
\doteq0.0267,\quad 
\sigma^2_f
\doteq0.0430,\quad 
\sigma_g^2
\doteq 0.0195. 
\]
\end{thm}

See \cref{T_Ar,T_Af,T_Ag} below for the 
precise values of
$\mu$ and $\sigma^2$. 
Using \cite{CD18} discussed above, 
we find that about $N^*=e^{\mu n+\sigma\sqrt{n}}/F_{n,1}$ samples 
are sufficient to well approximate $F_{n,1}$. 
For example, only about 194, 1520 and 75 samples are required 
for 
$F_{200,1}\doteq 4.5397\times 10^{41}$
by algorithms $\A_r$, 
$\A_f$ and $\A_g$. 
See \cref{S_compare} for more data.

In particular, we verify a conjecture of Don Knuth 
\cite{DKpc} that 
about $e^{c_r n+O(\sqrt{n})}$, where $c_r\doteq0.013$, samples 
are needed to approximate $F_{n,1}$
using $\A_r$.
Indeed, by \cref{T_Ar}, and since $F_{n,1}\sim\varphi^{n+1}/\sqrt{5}$, where 
$\varphi=(1+\sqrt{5})/2$, we find
\[
c_r=
\frac{1}{5}(\frac{13}{6}-\frac{2}{\sqrt{5}})\log{2}
+\frac{1}{5}(1+\frac{1}{\sqrt{5}})\log{3}
-\log\varphi
\doteq
0.013143.
\]

\subsection{Discussion}
We conclude with a series of  remarks. 

\subsubsection{Comparing algorithms} By the results above we see that, in the case of Fibonacci matchings, the random order algorithm 
$\A_r$ performs significantly better than 
$\A_f$, which matches ``from the top.'' 
This was not clear to us at the outset, and was an early motivating question. 
It would appear that the reason for this is that typically 
in several steps of $\A_r$ there 
are three choices for $\pi(i)$, whereas
in $\A_f$ there are only ever at most two. 
The greedy algorithm $\A_g$ capitalizes on this, and outperforms the others. 
Note that, although the means $\mu_r$, $\mu_f$ and $\mu_g$ 
are roughly comparable, 
the variances 
$\sigma_r^2$, $\sigma_f^2$ and $\sigma_g^2$ differ significantly.

The intuition behind $\A_g$ (sequentially matching the nearest unmatched vertex
amongst those of largest degree) may lead to  
useful strategies
for other classes of matchings, or other similar combinatorial tasks
(where, depending on the situation, 
``largest degree'' might be replaced with e.g.\ ``most spead
out choice'').

\subsubsection{Exponential running times} Further to the exponential running 
times of our algorithms, 
we mention here that we 
also consider variants 
in \cref{S_Aopt} which make non-uniform
decisions about how to match indices. 
Such non-uniform choices are crucial in the application of sequential 
importance sampling to estimating the number of contingency tables
with fixed row/column sums \cite{CDHL05} or the number of graphs 
with a given degree sequence \cite{BD10}. 
The present paper gives the first set of cases where such improvements can be proved. 

More specifically, we find ``almost perfect'' versions $\A^*$, 
for which $\var_u T(\pi)=O(1)$. 
Such 
algorithms 
require only $O(1)$ samples
to well approximate $F_{n,1}$.   
For example, in the case of the fixed order algorithm 
$\A_f$, we obtain $\A_f^*$ by in step $i$
setting $\pi(i)=i$ with probability $1/\varphi$
and $\pi(i)=i+1$ with probability $1/\varphi^2$ (unless
in the previous step $\pi(i-1)=i$, in which case 
the matching $\pi(i)=i-1$ is forced). 
Recall that 
$F_{n,1}\sim\varphi^{n+1}/\sqrt{5}$, so 
$1/\varphi$ and $1/\varphi^2$ are the asymptotic 
proportions of Fibonacci matchings with $\pi(1)=1$ and $\pi(1)=2$. 
The optimal probabilities for other ``almost perfect'' algorithms 
$\A^*$ are found by similar considerations.

Of course, in these examples we have used 
the asymptotics
of $F_{n,1}$, which in practice we would instead 
be attempting to approximate. Nonetheless, these observations 
suggest that {\it adaptive}
versions of our algorithms (using e.g.\  
the cross-entropy method~\cite{BKMR05})  
might achieve sub-exponential 
running times.

\subsubsection{CLTs for random order} Although    
Fibonacci matchings ${\mathcal F}_{n,1}$
and the random
matching scheme $\A_r$
are quite simple and natural,  
we do not see a clear way of proving a central limit theorem
for $\A_r$
by standard techniques. 
Instead, we apply the theory of 
Neininger and R\"{u}schendorf~\cite{NR04}
for probabilistic recurrences
associated with randomized divide-and-conquer algorithms, 
often occurring in computer science contexts. 
This work develops an observation of  
Pittel \cite{P99} that random combinatorial structures 
with ``nearly linear'' 
mean and variance, and a 
recurrence for the associated moment generating functions,  
are often asymptotically normal. These methods  
may be useful for analyzing the performance of 
importance sampling algorithms in other applications. 

The main difficulty in proving versions of \cref{T_main} for 
$\A_r$ on matchings 
in ${\mathcal F}_{n,t}$ and ${\mathcal D}_{n,d}$, more generally, is calculating 
the mean and variance of $\log T_r(\pi)$. The computations 
become quite involved. 
The theory of \cite{NR04} on the other hand is considerably robust. 

\subsubsection{Related works} Importance sampling was recently 
applied to Fibonacci and related matchings  
in \cite{CDG18}. 
This work takes a combinatorial approach
to obtaining, amongst other results, 
the asymptotics for the mean and variance of $\log T(\pi)$
under certain {\it fixed order} matching schemes. 
The present 
work, on the other hand, uses probabilistic arguments
to obtain central limit theorems for $\log T(\pi)$. 
This allows for a more accurate evaluation 
of the performance of these algorithms. 

A different approach to proving central limit theorems, 
using limit theory for inhomogenious Markov chains is being developed by Andy Tsao \cite{Andy}. 
These are competitive with the present approach in some cases, 
but also suffer from the same difficulty we have; 
it is often not straightforward  to compute the mean and variance of the limiting normals.

In some of our examples, generating functions for quantities of interest are available, 
either in closed form, or as solutions of differential equations. 
In these cases, methods of analytic combinatorics can give central limit theorems. 
For a worked example with $t$-Fibonacci matchings, 
see Melczer \cite{Melczer} Section 5.3.3. 

We also establish central limit theorems for {\it random order} matching schemes, 
which 
seem inaccessible by the methods in \cite{CDG18,Andy}. 
Prior to our work, the best upper bound
for the required sample size $N$  
in the case of sampling from ${\mathcal F}_{n,1}$ using $\A_r$  
was $N\le 6^{n/3}/F_{n,1}$, coming 
from the recent work \cite{D18}, where 
Br\`egman's inequality  \cite{B73} from matching theory
is used to show that for any bipartite graph
\[
N \le \frac{1}{M_n}\prod_{i=1}^n d_i!^{1/d_i}.
\]
However, this general bound 
only 
gives
that for algorithm $\A_r$, a sample of size 
$1.6446\times 10^{10}$ 
is sufficient to well approximate $F_{200,1}$. 
However \cref{T_main} shows that far fewer, only about 194 samples, suffice.

\subsubsection{MCMC alternatives} Another approach to computing averages over bipartite matchings is to use the 
Markov chain Monte Carlo algorithms of
\cite{MSV04,DJM,BSVV08}. These have provable polynomial running times.
Unfortunately, since the focus in the computer science theory literature
is on general results, the running times are given as $O^*(n^7)$ (where $O^*$
 hides logarithmic factors and dependence on an error parameter $\eps$
controlling the accuracy of the algorithm). 
For our 
example of Fibonacci matchings, 
the greedy algorithm $\A_g$, for instance, outperforms 
$n^7$ within a range of practical interest, until about $n=4894$. 

We have here of course committed the offense of using a general bound on a specific 
problem. Indeed, for the special case of Fibonacci matchings, \cite{Andy,Olena,B12} shows that the 
Markov chain mixes in order $n\log n$. 
Alas, the $n^7$ bound is all that is generally available; 
tuning the results for specific graphs is hard work
and largely open mathematics. 
With this caveat, we will compare with the benchmark $n^7$ 
in what follows.

\subsection{Acknowledgments}

We thank Don Knuth, Ralph Neininger and Andy Tsao for helpful conversations. 
PD was supported in part by NSF Grant DMS-1608182.
BK was supported in part by an NSERC Postdoctoral Fellowship.

\section{Background}\label{S_2}

 This section gives additional background on 
matching theory
and  importance
 sampling,  
 as well as a brief review of related literature.

\subsection{Matching theory}\label{S_MT}

Matching, with its many attendant variations, is a basic topic of combinatorics.
The treatise of 
Lov\'{a}sz and Plummer~\cite{LP86} covers most aspects. While there are efficient 
algorithms for determining if a bipartite graph has a perfect matching,
Valiant~\cite{V79} showed that counting the number of perfect matchings
is $\#{\rm P}$-complete. We encountered the enumeration and 
equivalent sampling problems through statistical testing scenarios where
permutations with restricted positions appear naturally.
They appear in evaluating strategies in card guessing experiments
 \cite{CDGM81,DG81}. 
In \cite{DGH99}
bipartite matchings occur in testing association with truncated data.
Other statistical applications are surveyed in 
Bapat \cite{B90}. Matching
problems have had a huge resurgence because of their appearance in 
donor matching for organ transplants, medical school applications
and ride sharing \`a la Uber/Lyft where drivers are matched to customers, see  
e.g.\ \cite{MR92} and references therein. 

A host of approximation algorithms are available. These range 
from probabilistic limit theorems (e.g., the chance that $\pi(i)\neq i$
for all $i$) with many extensions, see \cite{BHJ92} through recent
work on stable polynomials \cite{B16}. As discussed above, the widely cited works  
\cite{MSV04,DJM,BSVV08} offer Markov chain Monte Carlo approximations which
provably work in $O^*(n^7)$ operations.

\subsection{Importance sampling}\label{S_IS}

Let ${\mathcal X}$ be a space with $\nu$ a probability measure on ${\mathcal X}$. 
For a real valued function on ${\mathcal X}$ with finite mean, let 
\[
I(f) = \int f(x)\nu(dx). 
\]
One wants to estimate $I(f)$, however this can be intractable, so one
calls upon a second measure $\mu\gg\nu$
that is ``easy to sample from'' with 
\[
\rho(x) = d\nu / d\mu.
\]
Then $I(f) =\int f(x) \rho(x)\mu(dx)$ so we may sample 
$x_1,x_2,\ldots, x_N$ from $\mu$ to estimate $I(f)$ by
\[
\hat I_N(f) = \frac{1}{N} \sum_{i=1}^N f(x_i) \rho(x_i). 
\]
For background and surveys with many variations
and examples, see e.g.\ \cite{Liu08,CD18}. 

In our examples, ${\mathcal X}={\mathcal M}_n$, the set of perfect
matchings in a bipartite graph 
of size $M_n=|{\mathcal M}_n|$, $\nu(\pi) = 1/M_n$
is the uniform distribution and $\mu(\pi) = P(\pi)$
a specified importance sampling distribution. 
Then $\rho(\pi) = T(\pi)/M_n$, where $T(\pi)=P(\pi)^{-1}$. 

We investigate two ways $N^v$ and $N^*$ 
of assessing the sample size $N$
required to well approximate $M_n$. 
The first is the classical criteria $N^v$ of choosing 
$N$ based on the variance
\[
\var \hat I_N(f) = \frac{1}{N} \var (f(x)\rho(x)).
\]
For the standard deviation of $\hat I_N$ to be $\ll I(f)$, we require 
$N\gg \var(f\rho)/I(f)^2$. 
In our examples, estimating $M_n$, we have $f(\pi)=M_n$
and $\var(f\rho)=\sum_\pi T(\pi)$, 
and so the criteria becomes
\begin{equation}\label{E_relvar}
N^v\gg 
\frac{1}{M_n^2}\E_\nu[T(\pi)^2]
=\frac{1}{M_n^2}\sum_\pi T(\pi), 
\end{equation}
since typically $[\E_\nu T(\pi)]^2\ll \E_\nu[T(\pi)^2]$. 

As eluded to above, a different approach to sample size determination is suggested in 
\cite{CD18}. 
Let us now state this result precisely. 
They argue that importance sampling estimators are 
notoriously long-tailed and that variance may be a poor indication of accuracy. 
Define the 
Kullback--Leibler divergence by 
\[
L=D(\nu|\mu) 
= \int\rho\log\rho d\mu
= \int\log\rho d\nu
=\E_\nu\log Y,
\]
where $Y=\rho(X)$ and $X$ has distribution $\nu$. 
The main result shows that (roughly) ``$N=e^L$ is necessary and 
sufficient for accuracy.'' 

\begin{thm}[{\cite{CD18} Theorem~1.1}]\label{T_CD}
Let $I$, $\hat I_N$ and $L$ be as defined above. 
\begin{enumerate}[nosep,label=(\roman*)]
\item 
For $f$ with $||f||_{2,\nu}=\sqrt{\E [f(Y)^2]}<\infty$, $N=e^{L+t}$ and $t>0$, we have that 
\[
\E|\hat I_N(f)-I(f)|\le ||f||_{2,\nu}[e^{-t/4}+2\P_\nu^{1/2}(\log Y>L+t/2)].
\]
\item 
Conversely, if $f\equiv 1$, $N=e^{L-t}$ and $t>0$, then for any $\delta\in(0,1)$, 
\[
\P_\nu (\hat I_N(f)>1-\delta)
\le e^{-t/2}+[\P_\nu(\log Y<L-t/2)]/(1-\delta).
\]
\end{enumerate}
\end{thm}

To explain, suppose that $||f||_{2,\nu}\le1$, e.g., $f$ is the indicator of a set. 
Then (i) says that if $N>e^{L+t}$ and $\log Y$ is concentrated about its mean
(recall $\E\log Y=L$), then $\hat I_N(f)$ is close to $I(f)$ with high 
probability (use Markov's inequality with (i)). 
Conversely, (ii) shows that if $N<e^{L-t}$ and $\log Y$ is concentrated about its mean, 
then $I(1)=1$, but there is only a small probability that $\hat I_N(1)$ is correct.

This theorem suggests that in our examples $N^*$ samples
are sufficient, where 
\begin{equation}\label{E_e^L}
N^*\gg e^{L+\sigma},\quad 
L =
\E_\nu\log \rho(\pi)=
\frac{1}{M_n}\sum_{\pi}\log(T(\pi)/M_n)
\end{equation}
and $\sigma^2=\var_\nu \log\rho(\pi)$. 

Examples in \cite{CDG18} show that \eqref{E_relvar} and \eqref{E_e^L}
can lead to very different estimates of the required sample size. 
Of course, the required $N$ must be estimated, either
by analysis (as done in \cite{CDG18}) or by sampling
(see \cite{CD18}~Section~4).
The concentration of $\log Y$ must also be established. In \cite{CDG18}
this is done by computing the variance of $\log Y$. In 
this work, we obtain concentration 
in a number of examples
by proving a central limit theorem
for $\log Y$.

\section{Fibonacci matchings ${\mathcal F}_{n,1}$}\label{S_Fib}

In this section, we consider sampling 
 Fibonacci matchings 
 \[
 {\mathcal F}_{n,1}=\{\pi\in S_n:|\pi(i)-i|\le1\}
 \] 
 by various importance sampling
schemes. 
Of course, for this test case we in fact 
know that $F_{n,1}=|{\mathcal F}_{n,1}|=F_{n+1}$
and other detailed information about such matchings. 
For instance, in a uniform Fibonacci matching,
it is more likely for a given index $0<i<n$ to be 
involved in a transposition 
$\pi(i)\in\{i\pm1\}$ than to be a 
fixed point $\pi(i)=i$. 
Indeed, 
the 
asymptotic proportions of Fibonacci matching 
with $\pi(i)=i$ and $\pi(i)\in\{i\pm1\}$ are 
$1/\sqrt{5}$ and $2/\varphi\sqrt{5}$. 
On the other hand, it is more likely for indices $1$ and $n$
to be  a fixed point. 
In practice, however, the precise number of perfect matchings and other
related information is often 
unknown and to be estimated. 
It is thus natural to ask how accurately $F_{n,1}$
can
be estimated by an importance sampling algorithm, which rather than using such detailed
information, instead in each step 
matches an index $i\in[n]$ with another index
uniformly at random amongst the remaining 
allowable options. 

\begin{figure}[h!]
\centering
\includegraphics[scale=1]{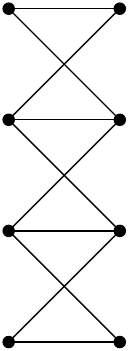}
\caption{Perfect matchings in this graph (see \cref{F_fib_n4})
correspond to Fibonacci permutations in ${\mathcal F}_{4,1}$. 
}
\end{figure}

We analyze three such algorithms.  
First, in \cref{S_Ar} we consider algorithm
$\A_r$ which matches indices in uniformly 
random order. Then in \cref{S_Af} we consider
algorithm $\A_f$ which matches in the 
fixed order $1,2,\ldots,n$. 
It turns out that $\A_r$
outperforms $\A_f$, 
as the variance of $\log T(\pi)$ is much smaller. 

In \cite{CDG18} the algorithm which 
matches in order $2,5,8,\dots$ (and then
fills in the rest in order  
$1,3,4,\ldots$) is analyzed. This algorithm performs
better than $\A_f$, since in many of its steps
there are 3 choices for $\pi(i)$, whereas in  $\A_f$
there are always at most 2. 
In \cref{S_Ag} we analyze a related 
algorithm $\A_g$ which improves on all of the above. 
Roughly speaking, this algorithm first matches
2, and then in each subsequent step matches the 
second smallest index whose matching has yet to be 
determined. This further increases the number of steps
with 3 choices.  

We let $N^*(n)$ denote the required sample size
given by the  
Kullback--Leibler \eqref{E_e^L} criterion. 
The 
main results of this section,  
\cref{T_Ar,T_Af,T_Ag}, combined with 
\cref{T_CD} 
imply, for instance, that 
\begin{equation}\label{E_N*1}
N^*_f(200)\doteq1520,\quad 
N^*_r(200)\doteq194,\quad 
N^*_g(200)\doteq75. 
\end{equation}
Thus all three algorithms result in small sample sizes, relative
to the total number of matchings $F_{200,1}\doteq 4.5397\times 10^{41}$. 
Bounds on sufficient sample size $N^v$ given  
by relative variance considerations \eqref{E_relvar}
are calculated in \cref{S_relvar} below. 
See \cref{S_compare} for a comparison of 
$N^*$ and $N^v$ under algorithms 
$\A_r$, $\A_f$ and $\A_g$
for various values of $n$.

As discussed in \cref{S_intro}, 
limit theory  \cite{NR04} for distributional recurrences
of divide-and-conquer type is 
a key element  
in our proofs.  
We next briefly 
recall this theory, before turning to the analysis
of the algorithms described above.

\subsection{Concentration, via  probabilistic recurrence}\label{S_NR}

A key observation for the analysis of several of our 
algorithms is that
the corresponding random variable $Y_n=\log T(\pi)$, under the uniform 
distribution $u$ on a set of 
matchings $\pi$ depending on $n$
(in this section, $\pi\in{\mathcal F}_{n,1}$), 
satisfies 
a distributional divide-and-conquer type recurrence relation
of the form 
\begin{equation}\label{E_NR}
Y_n 
\overset{d}{=} 
Y^{(1)}_{I_1^{(n)}}+Y^{(2)}_{I_2^{(n)}}
+b_n. 
\end{equation}
The function $b_n$ is the {\it toll function,}
giving the ``cost'' of splitting the problem of determining $Y_n$
into sub-problems of sizes $I_1^{(n)}$ and $I_2^{(n)}$. 
The special case of recursions of Quicksort type, where $I_1^{(n)}=U_n$
is uniform on $\{0,1,\ldots,n-1\}$ and $I_2^{(n)}=n-I_1^{(n)}-1$, 
was first studied in \cite{HN02}. Therein it is shown that 
$b_n=\sqrt{n}$ is roughly the threshold at which point larger
toll functions can lead to non-normal limits. 
General recursions of the form \eqref{E_NR} 
are analyzed in \cite{NR04}. 
We use the following special case of their results.  

\begin{lem}[\cite{NR04} Corollary~5.2]
\label{L_NR} 
Suppose $Y_n$ 
is $s$-integrable, $s>2$, and satisfies   
\eqref{E_NR} 
where $(Y_n^{(1)}),(Y_n^{(2)}),(I^{(n)}_1,I^{(n)}_2,b_n)$ are independent, 
all $Y_i^{(j)}\overset{d}{=}Y_i$, 
and $I^{(n)}_1,I^{(n)}_2\in\{0,1,\ldots,n\}$.
Suppose that, for some positive functions $f$ and $g$,   
\[\E Y_n = f(n)+o(g^{1/2}(n)),\quad  
\var Y_n=g(n)+o(g(n)).\]
Further assume that, for some $2<s\le3$ and $j=1,2$, 
\[
\frac{b_n-f(n)-f(I_1^{(n)})-f(I_2^{(n)})}{g^{1/2}(n)}\to0,\quad 
\frac{g^{1/2}(I_j^{(n)})}{g^{1/2}(n)}\to A_j
\]
in $L_s$, and that   
\[
A_1^2+A_2^2=1,\quad \P(\max\{A_1,A_2\}=1)<1. 
\]
Then, as $n\to\infty$, 
\[
\frac{Y_n-f(n)}{g^{1/2}(n)}\xrightarrow{d} N(0,1). 
\]
\end{lem}

Note that the $I^{(n)}_1$ and $I^{(n)}_2$
need not sum to $n$. 
In all of our applications of this result, the technical conditions
follow easily. In particular, we 
will always have 
$\E Y_n=\mu n+O(1)$ and $\var Y_n=\sigma^2 n+O(1)$, 
for some constants $\mu$ and $\sigma^2$, and $b_n=O(1)$. 

We note here that recent work
\cite{N15} shows that, in the special case of Quicksort, 
the independence 
condition in \cref{L_NR} can be relaxed given
suitable control of third absolute moments. 
It seems plausible \cite{NRpc} that for very small
toll functions $b_n=O(1)$, 
\cref{L_NR} holds with dependence on $b_n$
under workable conditions. 
This might, for instance, allow us to apply 
this theory also in the greedy case 
$\A_g$ (\cref{S_Ag} below) for which we instead
use renewal theory arguments to 
establish a central limit theorem.

\subsection{Random order algorithm $\A_r$}\label{S_Ar}

We first consider algorithm
$\A_r$ which matches indices in 
uniformly random order. This algorithm is the most 
naturally suited to the results described 
in \cref{S_NR}. 

In the first step of $\A_r$, an index $i$
is selected uniformly at random 
from those in $[n]$. Then $\pi(i)$ is set
to some index $j\in\{i,i\pm 1\}\cap[n]$ 
uniformly at random. 
Note that if $\pi(i)=i+1$ then necessarily $\pi(i+1)=i$. 
Similarly, if $\pi(i)=i-1$ then necessarily $\pi(i-1)=i$. 
As such, the problem 
of determining $\pi$ then splits into two
independent sub-problems, 
namely determining $\pi$ on $[\min\{i,j\}-1]$
and  $[n]-[\max\{i,j\}]$. 

We establish the following central limit theorem
for $\log T(\pi)$. 
Note that to obtain the required concentration of 
$\log Y=\log (T(\pi)/F_{n,1})$ appearing in
\cref{T_CD}, we simply 
subtract the (deterministic) quantity $\log F_{n,1}$. 

\begin{thm}\label{T_Ar}
Let $P_r(\pi)$ be the probability of $\pi$
under the random order algorithm $\A_r$, where   
$\pi$ is uniformly random with respect to 
$u=1/F_{n,1}$ on 
Fibonacci permutations $\pi\in{\mathcal F}_{n,1}$. 
Put $T_r(\pi)=P_r(\pi)^{-1}$.
As $n\to\infty$,  
\[
\frac{\log T_r(\pi)-\mu_r n}{\sigma_r\sqrt{n}}\xrightarrow{d} N(0,1) 
\]
where
\[
\mu_r = 
(\frac{13}{6}-\frac{2}{\sqrt{5}})\frac{\log{2}}{5}
+(1+\frac{1}{\sqrt{5}})\frac{\log{3}}{5}
\doteq 0.4944
\]
and
\begin{multline*}
\sigma^2_r= 
(\frac{1049}{10\sqrt{5}}-\frac{361}{9})\frac{\log^2 2}{50}
-(\frac{1579}{5\sqrt{5}}-113)\frac{\log2\log3}{225}\\
+(\frac{131}{5\sqrt{5}}-7)\frac{\log^2 3}{100}
\doteq 0.0267.
\end{multline*}
\end{thm}

To prove this result, we first show that 
$\log T_r$ has asymptotically
linear mean and variance.

\begin{lem}\label{L_Ar}
As $n\to\infty$, 
\[
\E_u \log T_r(\pi)
= 
\mu_r n +O(1),\quad 
\var_u \log T_r(\pi)
=
\sigma^2_rn+O(1)\]
where $\mu_r$ and $\sigma_r^2$ are as in \cref{T_Ar}.  
\end{lem}

More exact approximations are given in the proof, 
although those stated above are sufficient for our purposes.  

Before turning to the proof of this lemma, 
we first show how our main result follows by this 
and \cref{L_NR}. As mentioned in \cref{S_intro}, we do not see
how to prove this result easily by more standard methods.

\begin{proof}[Proof of \cref{T_Ar}]
Let $Y_n=\log T_r(\pi)$ for $\pi\in{\mathcal F}_{n,1}$ 
uniformly random
with respect to $u=1/F_{n,1}$. 
For convenience, set $F_{n,1}=0$
for all $n<0$.

As discussed already, algorithm 
$\A_r$ selects a uniformly random $i\in[n]$ 
and matches $i$ with one of $j\in\{i,i\pm1\}$ uniformly at random. 
Given this, it then matches the indices in 
$[\min\{i,j\}-1]$
and $[n]-[\max\{i,j\}]$ independently
and in a uniformly random order. Since there are no 
edges between these sets in a Fibonacci matching, 
we may 
view the matchings of these sets 
by $\A_r$
as occurring separately
and also in uniformly random order. 
Therefore
\[
Y_n
\overset{d}{=}
 Y^{(1)}_{I_1^{(n)}}+Y^{(2)}_{I_2^{(n)}}+b_n, 
\]
where, for $i\in[n]$, 
\begin{multline*}
\P((I_1^{(n)},I_2^{(n)})=(i_1,i_2))\\
=
\frac{1}{nF_{n,1}}\begin{cases}
F_{i-2,1}F_{n-i,1}& 
(i_1,i_2)=(i-2,n-i)\\ 
F_{i-1,1}F_{n-i,1}& 
(i_1,i_2)=(i-1,n-i)\\
F_{i-1,1}F_{n-i-1,1}& 
(i_1,i_2)=(i-1,n-i-1). 
\end{cases}
\end{multline*}
Here $F_{i-2,1}F_{n-i,1}/F_{n,1}$,
$F_{i-1,1}F_{n-i,1}/F_{n,1}$ and
$F_{i-1,1}F_{n-i-1,1}/F_{n,1}$ are the probabilities
that $\P_u(\pi(i)=j)$ for $j=i-1$, $i$ and $i+1$. 
The extra term $b_n=O(1)$, is equal to $\log2$
if $i\in\{1,n\}$ and $\log3$ if $i\in[n]-\{1,n\}$, as there
are 2 and 3 choices for $\pi(i)$ in these cases and 
$\A_r$ choses from them uniformly at random.
By \cref{L_Ar,L_NR} the result follows, noting that 
$I_1^{(n)}\overset{d}{=}U_n+O(1)$ and 
$I_2^{(n)}\overset{d}{=} n-U_n+O(1)$ for $U_n$ uniform on $[n]$. 
\end{proof}

Finally, to complete the proof of 
\cref{T_Ar}, 
we analyze the mean and variance 
of $\log T(\pi)$. We give a detailed proof 
of this result for $\A_r$.  
Similar results for other algorithms analyzed below  
will be discussed more briefly.

\begin{proof}[Proof of \cref{L_Ar}]
For $\pi$ selected with respect
to $u=1/F_{n,1}$ and $t\in \R$, put  
\[
x_n=F_{n,1}\E_u(e^{t \log T_r(\pi)})
=F_{n,1}\E_u[T_r(\pi)^t].
\] 
Denote the associated generating function by 
\[
X(z,t)=\sum_{n=0}^\infty x_n z^n.
\] 
Note that $X$ is the Hadamard product of the generating
functions for the sequences $F_{n,1}$ and 
$\E_u[T_r(\pi)^t]$, the latter of which is 
the moment generating 
function of the random
variable $\log T_r(\pi)$ under $u=1/F_{n,1}$. 
Therefore, for $k\ge0$, 
\begin{equation}\label{E_mk}
{\mathbb E}_u[(\log T_r(\pi))^k]
=
\frac{1}{F_{n,1}}
[z^n]\frac{\partial^k}{\partial t^k}X(z,0).
\end{equation}

We obtain a recurrence for $x_n$ as follows. 
Observe that trivially $x_0=x_1=1$. 
For $n\ge2$, 
we argue as follows. 
Algorithm $\A_r$ first selects 
$i\in[n]$ uniformly at random. 
If $i\in\{1,n\}$ then it matches $i$ correctly with $\pi(i)$ with probability 1/2, 
and otherwise, if $i\in[n]-\{1,n\}$, this happens with probability $1/3$. 
Since $\pi$ is uniformly random, note that we have  
$\pi(1)=1$ and $\pi(1)=2$ with probabilities 
$F_{n-1,1}/F_{n,1}$
and $F_{n-2,1}/F_{n,1}$, and a similar situation 
in the case $i=n$. 
On the other hand, for $0<i<n$, we have 
$\pi(i)=i-1$, $\pi(i)=i$ and $\pi(i)=i+1$
with probabilities 
$F_{i-2,1}F_{n-i,1}/F_{n,1}$, 
$F_{i-1,1}F_{n-i,1}/F_{n,1}$ and 
$F_{i-1,1}F_{n-i-1,1}/F_{n,1}$. 
Given the first step of $\A_r$ is successful, 
the algorithm then matches 
indices in $[\min\{i,\pi(i)\}-1]$
and $[n]-[\max\{i,\pi(i)\}]$ independently. 
Therefore, for $n\ge2$, 
\begin{multline*}
n x_n=2^{t+1}(x_{n-1}+x_{n-2})
+3^t\sum_{i=2}^{n-1}(
x_{i-2}x_{n-i}+x_{i-1}x_{n-i}+x_{i-1}x_{n-i-1}
)\\
=2(2^{t}-3^t)(x_{n-1}+x_{n-2})
+3^t(
\sum_{i=0}^{n-1}x_{i}x_{n-i-1}
+2
\sum_{i=0}^{n-2}x_{i}x_{n-i-2}
).
\end{multline*}
Summing over $n\ge2$
(and recalling $x_0=x_1=1$),
we find  
\begin{equation}\label{E_Xr}
\frac{\partial}{\partial z}X(z,t)-1\\
=2(2^t-3^t)[(1+z)X(z,t)-1]
+3^t[ (1+2z)X(z,t)^2-1].
\end{equation}
It follows immediately that 
\[
X(z,0)=\frac{1}{1-z-z^2}=\sum_{n=0}^\infty F_{n,1}z^n,
\]
as to be expected.
Differentiating \eqref{E_Xr} with respect to $t$ and 
 setting $t=0$, we obtain 
\begin{multline*}
\frac{\partial^2}{\partial z\partial t}X(z,0)
=2(\log{2}-\log{3})((1+z)X(z,0)-1)\\
+(\log{3})((1+2z)X(z,0)^2-1)+2(1+2z)X(z,0)\frac{\partial}{\partial t}X(z,0). 
\end{multline*}
Noting $(\partial/\partial t)X(0,0)=0$, it follows that  
\begin{multline}\label{E_dt}
(1-z-z^2)^2
\frac{\partial}{\partial t}X(z,0)
\\=
z^3(z^2+5z+5)\frac{\log3}{5} - 
z^2(12z^3+45z^2+20z-60)\frac{\log2}{30}. 
\end{multline}
Similarly, differentiating \eqref{E_Xr} twice with respect to $t$ and setting $t=0$,
we find that  
\begin{multline*}
\frac{\partial^3}{\partial z\partial t^2}X(z,0)=
2(\log^2 2-\log^2 3)((1+z)X(z,0)-1)\\
+4(\log{2}-\log{3})(1+z)\frac{\partial}{\partial z}X(z,0)
+\log^2 3((1+2z)X(z,0)^2-1)\\
+4(\log{3})(1+2z)X(z,0)\frac{\partial}{\partial z}X(z,0)
+2(1+2z)(\frac{\partial}{\partial z}X(z,0))^2\\
+2(1+2z)X(z,0)\frac{\partial^2}{\partial t^2}X(z,0). 
\end{multline*}
Noting $(\partial^2/\partial t^2)X(0,0)=0$, 
it follows
\begin{multline}\label{E_dtt}
(1-z-z^2)^3\frac{\partial^2}{\partial t^2}X(z,0)\\
= 
-z^3(z^7+5z^6-5z^5-50z^4-60z^3+10z^2-50)\frac{\log^2 3}{50}\\ 
+z^4(18z^6+90z^5-20z^4-710z^3-955z^2+435z+1125)\frac{\log2\log3}{225}\\
-z^2(36z^8+180z^7+25z^6-1250z^5-1745z^4+1065z^3+1875z^2-900)\frac{\log^2 2}{450}.
\end{multline}
 
Recall that 
\[
F_{n,1}=F_{n+1}=
(\varphi^{n+1}-\hat \varphi^{n+1})/\sqrt{5}
\sim\varphi^{n+1}/\sqrt{5},
\]
where $\varphi=(1+\sqrt{5})/2$ and $\hat\varphi=-1/\varphi$. 
Decomposing into partial fractions, 
\[
\frac{1}{(1-z-z^2)^2}
=
\sum_{x\in\{\varphi,\hat\varphi\}}
\frac{1}{5}[
\frac{1}{(z-1/x)^2}-\frac{2}{5}\frac{1+2/x}{z-1/x}
]
\]
and
\[
\frac{1}{(1-z-z^2)^3}
=
\sum_{x\in\{\varphi,\hat\varphi\}}
\frac{1}{25}[
-\frac{1+2/x}{(z-1/x)^3}
+\frac{3}{(z-1/x)^2}
-\frac{6}{5}\frac{1+2/x}{z-1/x} 
]. 
\]
Therefore
\begin{multline}\label{E_Fib2}
\frac{1}{F_{n,1}}[z^n]\frac{1}{(1-z-z^2)^2}
\sim 
\frac{1}{F_{n,1}}[z^n]
\frac{1}{5}
[\frac{1}{(z-1/\varphi)^2}
-\frac{2}{5}\frac{1+2/\varphi}{z-1/\varphi}]\\
\sim
\frac{1}{\sqrt{5}}[(n+1)\varphi +\frac{2}{5}(1+\frac{2}{\varphi})]
\end{multline}
and
\begin{multline}\label{E_Fib3}
\frac{1}{F_{n,1}}[z^n]\frac{1}{(1-z-z^2)^3}\\
\sim 
\frac{1}{F_{n,1}}[z^n]
\frac{1}{25}
[
-\frac{1+2/\varphi}{(z-1/\varphi)^3}
+\frac{3}{(z-1/\varphi)^2}
-\frac{6}{5}\frac{1+2/\varphi}{z-1/\varphi}
]\\
\sim
\frac{1}{5\sqrt{5}}
[
\frac{(n+2)(n+1)}{2}(1+\frac{2}{\varphi})\varphi^2
+3(n+1)\varphi
+\frac{6}{5}(1+\frac{2}{\varphi})
]
\end{multline}
where 
$\sim$ denotes equality up terms of size 
$O(n^2(\hat\varphi/\varphi)^n)=O((n/\varphi^n)^2)=o(1)$.

Finally, we conclude as follows. Let
\[
A_n=
\frac{1}{\sqrt{5}}[(n+1)\varphi +\frac{2}{5}(1+\frac{2}{\varphi})]
\sim
\frac{1}{F_{n,1}}[z^n]\frac{1}{(1-z-z^2)^2}
\]
denote the right hand side of \eqref{E_Fib2}. 
Combining \eqref{E_mk},  
\eqref{E_dt} and \eqref{E_Fib2}, we find that 
\begin{multline*}
{\mathbb E}_u \log T_r(\pi)\sim
[\frac{A_{n-5}}{\varphi^5}
+5\frac{A_{n-4}}{\varphi^4}
+5\frac{A_{n-3}}{\varphi^3}
]\frac{\log3}{5}\\
-[
12\frac{A_{n-5}}{\varphi^5}
+45\frac{A_{n-4}}{\varphi^4}
+20\frac{A_{n-3}}{\varphi^3}
-60\frac{A_{n-2}}{\varphi^2}
]\frac{\log2}{30}. 
\end{multline*}
Straightforward, although tedious, algebra then shows 
\[
{\mathbb E}_u\log T_r(\pi) 
\sim 
\mu_rn+
\frac{1}{25}  [(\frac{1}{\sqrt{5}}-\frac{16}{5})\log{3}+(\frac{71}{3}-\frac{185}{6\sqrt{5}})\log{2}]
\]
giving the first claim of the lemma. 

Similarly, 
by \eqref{E_mk},  
\eqref{E_dtt} and \eqref{E_Fib3}
we find  
\begin{multline*}
{\mathbb E}_u[(\log T_r(\pi))^2] 
-(\mu_rn)^2\\
\sim[(\frac{43}{5\sqrt{5}}-31)\frac{\log^2{3}}{100}
+(44-\frac{937}{5\sqrt{5}})\frac{\log2\log3}{225}
+(\frac{1223}{10\sqrt{5}}+97)\frac{\log^2 2}{450}]n\\
+(\frac{765}{\sqrt{5}}+547)\frac{\log^2 3}{2500}
+(\frac{2}{3}-\frac{2795}{\sqrt{5}})\frac{\log2\log3}{1875}
+(\frac{5935}{\sqrt{5}}-\frac{15983}{9})\frac{\log^2 2}{2500}. 
\end{multline*}
and so   
\begin{multline*}
\var_u \log T(\pi)
\sim 
\sigma^2_r n+
(\frac{1405}{\sqrt{5}}-497)\frac{\log^2 3}{2500}
+(\frac{7373}{9}-\frac{2155}{\sqrt{5}})\frac{\log2\log3}{625}\\
+(\frac{105955}{4\sqrt{5}}-10748)\frac{\log^2 2}{5625}
\end{multline*}
finishing the proof. 
\end{proof}

\subsection{Fixed order algorithm $\A_f$}\label{S_Af}

Next, we turn to the case of algorithm $\A_f$,
which recall matches in the fixed order
$1,2,\ldots,n$. 
Applying 
\cref{L_NR} in this case 
is perhaps less natural, but nonetheless establishes
the following central limit theorem. 

\begin{thm}\label{T_Af}
Let $P_f(\pi)$ be the probability of $\pi$
under the fixed order algorithm $\A_f$, where   
$\pi$ is uniformly random with respect to 
$u=1/F_{n,1}$ on 
Fibonacci permutations $\pi\in{\mathcal F}_{n,1}$. 
Put $T_f(\pi)=P_f(\pi)^{-1}$.
As $n\to\infty$,  
\[
\frac{\log T_f(\pi)-\mu_f n}
{\sigma_f\sqrt{n}}\xrightarrow{d} N(0,1) 
\]
where
\[
\mu_f=
\frac{1}{2}(1+\frac{1}{\sqrt{5}})\log 2\doteq0.5016,\quad
\sigma^2_f
=\frac{1}{5\sqrt{5}}\log^2 2\doteq0.0430. 
\]
\end{thm}

We note here that in \cite{D18} it is observed that 
this result follows by the results in \cite{DGH99}. 
Specifically,  it is observed that 
$P_f(\pi)=1/2^{n-k-1}$, where $k=k(\pi)$ is the number of 
transpositions in $\pi$ not counting $(n,n-1)$. 
This is because each time $\A_f$ matches $\pi(i)=i+1$,
the next matching $\pi(i+1)=i$ is made deterministically. 
The transposition $(n,n-1)$, however, does not contribute to $P_f(\pi)$,  
since the last 
step of algorithm $\A_f$ is always forced, irrespective of $\pi$. 
In \cite{DGH99}, several ways of establishing a central limit theorem for $k$ 
(for uniform matchings $\pi\sim u$)
are discussed. In particular (see \cite{DGH99}) 
\[
\E_u k(\pi)=\frac{1}{2}(1-\frac{1}{\sqrt{5}})n+O(1),
\quad
\var_u k(\pi)=\frac{1}{5\sqrt{5}}n+O(1).
\]
This leads to a proof of \cref{T_Af}. 

To point out yet another proof along these lines, 
we note that once the asymptotic linearity of $\E_u k(\pi)$
and $\var_u k(\pi)$ are established, it is easy to derive 
a central limit theorem for $Y_n=k(\pi)$ via \cref{L_NR}, noting that 
\[
Y_n \overset{d}{=} 
Y^{(1)}_{I_1^{(n)}}+Y^{(2)}_{I_2^{(n)}}+b_n
\] 
with  
\begin{multline*}
\P((I_1^{(n)},I_2^{(n)},b_n)=(i_1,i_2,b))\\
=
\frac{1}{F_{n,1}}
\begin{cases}
F_{k_n-2,1}F_{n-k_n,1}& 
(i_1,i_2,b_n)=(k_n-2,n-k_n,1)\\
F_{k_n-1,1}F_{n-k_n,1}& 
(i_1,i_2,b_n)=(k_n-1,n-k_n,0)\\
F_{k_n-1,1}F_{n-k_n-1,1}& 
(i_1,i_2,b_n)=(k_n-1,n-k_n-1,1)
\end{cases}
\end{multline*}
where $k_n=\lfloor n/2\rfloor$. 
To see this, observe these are the probabilities under $u$
that $\pi(k_n)$ is equal to $k_n-1$, $k_n$ or $k_n+1$. 
In the first case, given $\pi(k_n)=k_n-1$, note that $Y_n$
is distributed as $k(\pi_1)+k(\pi_2)+1$, where $\pi_1$
and $\pi_2$ are uniform Fibonacci permutations of 
sizes $k_n-2$ and $n-k_n$. The other cases are seen similarly. 

Let us also mention here that, through recent conversations with 
Andy Tsao \cite{Andy}, we have observed the following simple proof by renewal theory.  
Consider the renewal process $(N_t,t\ge0)$ whose inter-arrival times 
$X_i$ are equal to 1 and 2 with probabilities $1/\varphi$ and $1/\varphi^2$. 
Since  
$\P_u(\pi(1)=1)=F_{n-1,1}/F_{n,1}\doteq1/\varphi$
and $\P_u(\pi(1)=2)=F_{n-2,1}/F_{n,1}\doteq1/\varphi^2$, 
we see that $\log T_f(\pi)$ 
is well-approximated by $N_n\log2$.  
Hence \cref{T_Af} is essentially 
an immediate consequence of the central limit theorem
for renewal processes, noting that 
\[
\frac{1}{\E(X_1)}=\frac{1}{2}(1+\frac{1}{\sqrt{5}}),\quad
\frac{\var(X_1)}{\E(X_1)^3}=\frac{1}{5\sqrt{5}}. 
\]
We give a more detailed argument
of this kind for the greedy algorithm $\A_g$
in \cref{S_Ag} below, for which it seems \cref{L_NR} 
is not directly applicable. 

We conclude this section with a proof of \cref{T_Af}
by \cref{L_NR}, in order to further demonstrate its versatility. 
To motivate this proof,  
we note that the proof in \cite{D18} described above relies on the fact that 
for the fixed order algorithm $\A_f$, the probability $P_f(\pi)$ of obtaining 
a given permutation $\pi$ has a simple formula 
$1/2^{n-k-1}$. 
However, for many algorithms, such as the random 
order algorithm $\A_r$ studied in the previous section, $P(\pi)$ does not 
have a simple closed form. On the other hand, 
applying \cref{L_NR} does not require 
a formula for $P_f(\pi)$, rather 
only that $\A_f$ can be thought of as a 
divide-and-conquer algorithm of the form \eqref{E_NR}.

\begin{proof}[Proof of {\cref{T_Af}}]
The proof that $\log T_f(\pi)$ has asymptotically linear mean
and variance is similar
to \cref{L_Ar}. In fact, the calculations are simpler since
we can obtain the generating function
\begin{equation}\label{E_Xf}
X(z,t)
=
\frac{1+(1-2^t)z}
{1-2^tz(1+z)} 
\end{equation}
for the sequence of 
$x_n=F_{n,1}\E_u[T_f(\pi)^t]$ explictly
(whereas for $\A_r$ we had a differential equation
for $X$). 
Using this, it follows that 
\begin{equation}\label{E_momAf}
\E_u \log T_f(\pi)
=
\mu_f n 
+O(1),\quad 
\var_u \log T_f(\pi)
=\sigma^2_f n 
+O(1). 
\end{equation}
We refer to \cref{S_musigf} for further details. 

Our 
application of \cref{L_NR}, however, is less 
straightforward in the present case than in 
the random case \cref{T_Ar}.

Let $Y_n=\log T_f(\pi)$, where $\pi$ is
selected according 
to $u=1/F_{n,1}$. 
If we think of algorithm $\A_f$ in the most obvious  
way, as matching from top to bottom, we obtain  
\[Y_n  \overset{d}{=} Y^{(1)}_{I_1^{(n)}}+\log 2\]
with 
\[
\P(I_1^{(n)}=i_1)=
\frac{1}{F_{n,1}}\begin{cases}
F_{n-1,1}& 
i_1=n-1\\ 
F_{n-2,1}& 
i_1=n-2.
\end{cases}
\]
Here $F_{n-1,1}/F_{n,1}=\P_u(\pi(1)=1)$
and $F_{n-2,1}/F_{n,1}=\P_u(\pi(1)=2)$,
and $\log 2$ is the contribution to $\log T_f(\pi)$
from the probability that 
$\A_f$ matches $1$ correctly with $\pi(1)$.
However, if we think of $Y_n$ in this way, 
we clearly have $A_1=1$, and so \cref{L_NR}
does not apply.

Instead, to calculate $Y_n$ we informally speaking
divide at index $k_n = \lfloor n/2\rfloor$, that is, 
at ``the middle'' instead of at ``the top'' of $\pi$. 
In order for this division to result
in two independent and like problems, we 
apply a small trick. 
Consider a modified algorithm 
$\A_f^\eps$ on ${\mathcal F}_{n,1}\cup\{\eps\}$. Algorithm 
$\A_f^\eps$ operates in exactly the same way as $\A_f$,
except if in the very last step of the algorithm all indices in $[n-1]$
have been matched with indices in $[n-1]$. In this case, note that $\A_f$
deterministically matches $\pi(n)=n$. On the other hand, 
$\A_f^\eps$ does this with probability $1/2$, and otherwise, instead
of producing a matching $\pi\in{\mathcal F}_{n,1}$ returns $\eps$.  
Note that $\log T_f(\pi)=\log T_f^\eps(\pi)+O(1)$, since the two
quantities differ by at most $\log2$ for any $\pi$. Therefore it suffices to prove
a central limit theorem for $Y_n^\eps=\log T_f^\eps(\pi)$. 

We think of $\eps$ in this context as an ``error message''. 
This device allows for a constant toll function $b_n=\log2$
in the recursion below
for $Y_n^\eps$, whereas for $\A_f$ we cannot write a recursion 
for $Y_n$ in this way 
with all $Y_i^{(1)}\overset{d}{=}Y_i^{(2)}$. 
By the construction of $\A_f^\eps$, we have 
\[
Y_n^\eps \overset{d}{=} 
(Y^\eps_{I_1^{(n)}})^{(1)}+(Y^\eps_{I_2^{(n)}})^{(2)}+\log2
\] 
with  
\begin{multline*}
\P((I_1^{(n)},I_2^{(n)})=(i_1,i_2))\\
=
\frac{1}{F_{n,1}}
\begin{cases}
F_{k_n-2,1}F_{n-k_n,1}& 
(i_1,i_2)=(k_n-2,n-k_n)\\
F_{k_n-1,1}F_{n-k_n,1}& 
(i_1,i_2)=(k_n-1,n-k_n)\\
F_{k_n-1,1}F_{n-k_n-1,1}& 
(i_1,i_2)=(k_n-1,n-k_n-1). 
\end{cases}
\end{multline*}
To see this, note that the probabilities here correspond to 
whether $\pi(k_n)$ is equal to
$k_n-1$, $k_n$ or $k_n+1$ in a permutation $\pi\sim u$. 
Given the matching $\pi(k_n)$, the probability $P_f^\eps(\pi)$ 
is equal to the product $P_f^\eps(\pi_1)P_f^\eps(\pi_2)/2$, where 
$\pi_1$ and $\pi_2$ are uniform Fibonacci permutations of sizes
$\min\{k_n,\pi(k_n)\}-1$ and $n-\max\{k_n,\pi(k_n)\}$, 
and $1/2$ is the probability that $\A_f^\eps$ 
correctly matches $\min\{k_n,\pi(k_n)\}$ with $\pi(\min\{k_n,\pi(k_n)\})$.

By \eqref{E_momAf} and
since $k_n=k/2+O(1)$, the conditions of \cref{L_NR}
are clearly satisfied, giving the result. 
\end{proof}

\subsection{Greedy algorithm $\A_g$}\label{S_Ag}

In this section, we analyze algorithm 
$\A_g$, 
which matches 
indices $i\in [n]$ in a certain greedy order, so as to 
maximize the number of steps with 3 choices
for the value of $\pi(i)$. 

To describe $\A_g$ precisely, it is useful to think 
of indices 
as being either {\it matched} or {\it forced} by the algorithm. 
An index $i$ is matched if $\pi(i)$ is directly selected in a step 
of the algorithm, 
whereas we say that an index is forced if its matching is 
determined by previous matchings. 
In this terminology, 
in each step of algorithm $\A_g$, the vertex
of second smallest index, amongst those yet to be matched
or forced, is matched. 
More specifically, in the first step, $2$ is matched with 
one of $1,2,3$ uniformly at random. If $\pi(2)=1$, then 
$\pi(1)=2$ is forced, 
and so the algorithm matches $4$ in the next step with one of 
$3,4,5$ uniformly at random. 
Similarly, if instead $\pi(2)=2$ then $\pi(1)=1$ is forced, 
and so the algorithm matches $4$ in the next step
with one of $3,4,5$ at random; 
if $\pi(2)=3$ then $\pi(1)=1$
and $\pi(3)=2$ are forced, and so the algorithm 
matches $5$ in the next step with one of $4,5,6$ at random. 
The algorithm continues in this way 
until some $\pi\in S_n$ has been determined. 

\begin{thm}\label{T_Ag}
Let $P_g(\pi)$ be the probability of $\pi$
under the 
greedy 
algorithm $\A_g$, where   
$\pi$ is uniformly random with respect to 
$u=1/F_{n,1}$ on 
Fibonacci permutations $\pi\in{\mathcal F}_{n,1}$. 
Put $T_g(\pi)=P_g(\pi)^{-1}$.
As $n\to\infty$,  
\[
\frac{\log T_g(\pi)-\mu_g n}
{\sigma_g\sqrt{n}}\xrightarrow{d} N(0,1) 
\]
where
\[
\mu_g=\frac{1}{\sqrt{5}}\log{3}\doteq 0.4913,\quad 
\sigma_g^2=
(1-\frac{11}{5\sqrt{5}})\log^2 3\doteq 0.0195.  
\]
\end{thm}

Due to the nature of the greedy order there seems to be no obvious way
of writing $Y_n=\log T_g(\pi)$ in the form \eqref{E_NR}
while preserving the independence of 
$(Y_n^{(1)}),(Y_n^{(2)}),(I^{(n)}_1,I^{(n)}_2,b_n)$
required for the application of 
\cref{L_NR}. 
We instead give a proof via renewal theory, following the argument
outlined for the 
fixed algorithm $\A_f$ after the statement of \cref{T_Af} above.

\begin{proof}
The idea is to approximate $Y_n$ using the renewal process $(N_t,t\ge0)$ 
whose inter-arrival times $X_1$ are equal to 2 and 3 with probabilities 
$2/\varphi^2$ and $1/\varphi^3$. The reason is that if in the first step 
$\A_g$ matches $\pi(2)=1$
then $\pi(1)=2$ is forced; if $\pi(2)=2$ then $\pi(1)=1$ is forced; if $\pi(2)=3$ then $\pi(1)=2$ and $\pi(3)=2$
are forced. Each of these possibilities are equally likely under $\A_g$, and 
$\varphi^2$, $\varphi^2$ and $1/\varphi^3$ are the asymptotic probabilities that $\pi(2)$ is equal to 
1, 2 and 3 in a Fibonacci permutation $\pi\sim u$. Therefore we expect $Y_n$ to be 
well-approximated by $N_n\log3$. Note that, by the central limit theorem
for renewal processes, 
\begin{equation}\label{E_rtCLT}
\frac{N_t-t/\E(X_1)}{\sqrt{t\var(X_1)/\E(X_1)^3}}
\overset{d}{\to} N(0,1). 
\end{equation}
Since 
\[
\frac{1}{\E(X_1)}=\frac{1}{2}(1+\frac{1}{\sqrt{5}}),\quad
\frac{\var(X_1)}{\E(X_1)^3}=\frac{1}{5\sqrt{5}},
\]
the theorem is proved if we verify that 
\begin{equation}\label{E_YnCLT}
\frac{Z_n-n/\E(X_1)}{\sqrt{n\var(X_1)/\E(X_1)^3}}
\overset{d}{\to} N(0,1),
\end{equation}
where $Z_n=Y_n/\log3$. 

To this end, note that any $\pi\in {\mathcal F}_{n,1}$, can be viewed from ``top to bottom'', as a sequence of {\it blocks}
of the form 
\begin{itemize}[nosep]
\item $\pi(i-1,i)=i,i-1$, 
\item $\pi(i-1,i)=i-1,i$ or
\item $\pi(i-1,i,i+1)=i-1,i+1,i$ 
\end{itemize}
(except the very last block, which can be $\pi(n)=n$). 
Note that the indices $i$ are ones that need to be correctly matched by $\A_g$
in order to construct $\pi$ (the other indices being forced). 
In this terminology $Z_n$ is, up to an additive $O(1)$ error, 
simply the number of blocks in $\pi$. The $O(1)$ error accounts for 
unimportant issues at the ``bottom'' of $\pi$. 

Next, recall that in the proof of \eqref{E_rtCLT}, we 
do not in fact need that $(N_t)$ is a renewal process, 
rather only that 
\[
\frac{\sum_{i=1}^nX_i-n/\E(X_i)}{\sqrt{n\var(X_i)}}
\overset{d}{\to} N(0,1). 
\]

Finally, we make the connection with $(N_t)$. 
For a given $n$ consider the process $(X_i')$
where $X_1'$ is equal to 2 and 3 with probabilities $2F_{n-t_i-2,1}/F_{n,1}$
and $F_{n-t_i-3,1}/F_{n,1}$ where $t_i=\sum_{j<i}X_i'$. Let $N_n'=\max\{i:t_i\le n\}$.
Then the number of blocks in $\pi\sim u$ is equal in distribution to $N_n'+O(1)$. 
Hence $Z_n\overset{d}{=}N_n'+O(1)$. Noting that all 
\[
|F_{n-k,1}/F_{n,1}-1/\varphi^k|=O((\hat\varphi/\varphi)^n)=O(1/\varphi^{2n}),
\]
it is easy to see that 
\[
\frac{\sum_{i<N_n'}|X_i-X_i'|}{\sqrt{n}}
\overset{p}{\to} 0,
\]
and \eqref{E_YnCLT} follows. 
\end{proof}

\subsection{Almost perfect variants}\label{S_Aopt}

In this section, we observe that 
by adjusting the probabilities with which 
$\pi(i)$ is set
to $i$ or $i\pm1$ in algorithms 
$\A_f$ and $\A_g$
to agree with the asymptotic proportion
of Fibonacci matchings with the same value of $\pi(i)$, 
we obtain ``almost perfect'' 
version 
$\A_f^*$ and $\A_g^*$, for which 
$\log (T^*(\pi)/F_{n,1})$ has variance   
of $O(1)$ (under $u=1/F_{n,1}$).
As such, these algorithms require only $O(1)$
samples in order to well approximate $F_{n,1}$.

\subsubsection{Algorithm $\A_f^*$}

In algorithm $\A_f^*$ we match 
indices in order $1,2,\ldots,n$,  
setting $\pi(i)$ equal to $i$ with probability 
$1/\varphi$ and equal to $i+1$
with probability $1/\varphi^2$ (unless in the previous step $\pi(i-1)=i$ 
in which case we deterministically
set $\pi(i)=i-1$ as this is the only allowable option in this case). 
As explained above, the reason for this choice is that these are the 
asymptotic proportions of Fibonacci matchings with $\pi(1)=1$ and $\pi(1)=2$. 
This leads to the generating function
\[
X(z,t)
=
\frac{1+(1-\varphi^t)z}
{1-\varphi^t z(1-\varphi^{t}z)}
\]
for the sequence $x_n=F_{n,1}\E_u[T_f^*(\pi)^t]$. 
Using this, it can be shown (arguing as in \cref{S_musigf})
that 
\[
\E_u \log T_f^*(\pi)/n\to\log\varphi,\quad
\var_u \log T_f^*(\pi)=O(1). 
\]
Therefore, by \cref{T_CD}, $N^*(n)=O(1)$ samples 
are sufficient to well approximate
$F_{n,1}$ by this algorithm. 

We note here that a different heuristic is given in 
\cite{CDG18} Section 5.4, where the optimal $p^*=1/\varphi^2$
is derived in two ways by leaving the probability $p$ that in a given step $\pi(i)$ is set to 
$i+1$ as a free variable, and then computing the 
relative variance and $L$, and finally minimizing in $p$. In each case this 
leads to the same $p^*$.

\subsubsection{Algorithm $\A_g^*$}

Next, we note that $\A_g$ can similarly be modified
to obtain an ``almost perfect'' algorithm $\A_g^*$. As in  $\A_g$, we
start by 
matching 2 and then in each subsequent step
match the second smallest index whose matching
is yet to be determined, however now we match an 
index $i$ with $i-1$, $i$ and $i+1$ with probabilities 
$1/\varphi^2$, $1/\varphi^2$ and $1/\varphi^3$. 
This leads to the generating function
\[
X(z,t)
=
\frac{1+z+2(2^t-\varphi^{2t})z^2}
{1-\varphi^{2t}(2-\varphi^{t}z)z^2}, 
\]
for $x_n=F_{n,1}\E_u[T_g^*(\pi)^t]$, and using this it 
can be shown that 
\[
\E_u \log T_g^*(\pi)/n\to\log\varphi,\quad
\var_u \log T_g^*(\pi)=O(1). 
\]

\subsection{Relative variance}
\label{S_relvar}

Finally, we investigate the relative 
variance approach \eqref{E_relvar} to 
sample size determination for importance
sampling algorithms $\A_r$, $\A_f$ and $\A_g$. 
Recall that in the proofs of the asymptotic linear
variance of $\log T(\pi)$ for these algorithms, 
we made use of the generating  function $X(z,t)$
for the sequence of $x_n=F_{n,1}\E_u[T(\pi)^t]$. 
Let $Y(z)$ denote the generating function for the
sequence of $y_n=\E[T(\pi)^2]$, where the expectation
in this case is with respect to the importance
sampling measure $P(\pi)$. It is not difficult to
see that $Y(z)=X(z,1)$.  Using this observation, 
we analyze the bounds on the sufficient $N^v$ given 
by \eqref{E_relvar}, i.e., 
\[
N^v
=\frac{\E[T(\pi)^2]}{[\E T(\pi)]^2}
=\frac{1}{F_{n,1}^2}\sum_\pi T(\pi). 
\]

\subsubsection{$N^v$ for $\A_r$}

Let $Y(z)$ be the generating function 
for $y_n=\E[T_r(\pi)^2]$ under the measure
$P_r(\pi)$ on ${\mathcal F}_{n,1}$. 
By \eqref{E_Xr}, and observing that $Y(z)=X(z,1)$, we have
\[
Y'(z)
=-2(1+z)Y(z)+3(1+2z)Y(z)^2.
\] 
It follows that
\[
Y(z)=
\frac{1}{3}[1-e^{(1+z)^2}(2/3e-\int_{1}^{1+z}e^{-u^2}du)]^{-1}. 
\]
Basic calculus arguments show that the function 
$e^{(1+z)^2}(2/3e-\int_{1}^{1+z}e^{-u^2}du)$ is convex on $\R$
and minimized at $z\doteq-0.2115$, see \cref{F_Y} (we omit the details).   
Using this, we note that the function $Y(z)$ has singularities at 
$z_r\doteq0.3720$ and 
$z'\doteq-1.0079$, as can be verified numerically. 
Let 
\[
c_r=\lim_{z\to z_r}(z_r-z)Y(z)\doteq 0.1911. 
\]  
By \cite{FS09} Corollary~VI.1, as $n\to\infty$, 
\[
\E[T_r(\pi)^2]
=[z^n]Y(z)\sim c_rz_r^{-(n+1)}.
\]
Therefore, since $F_{n,1}=F_{n+1}\sim \varphi^{n+1}/\sqrt{5}$, 
\[
N^v_r\sim 
5c_r(z_r\varphi^2)^{-(n+1)}.
\]

\begin{figure}[h!]
\centering
\includegraphics[scale=0.25]{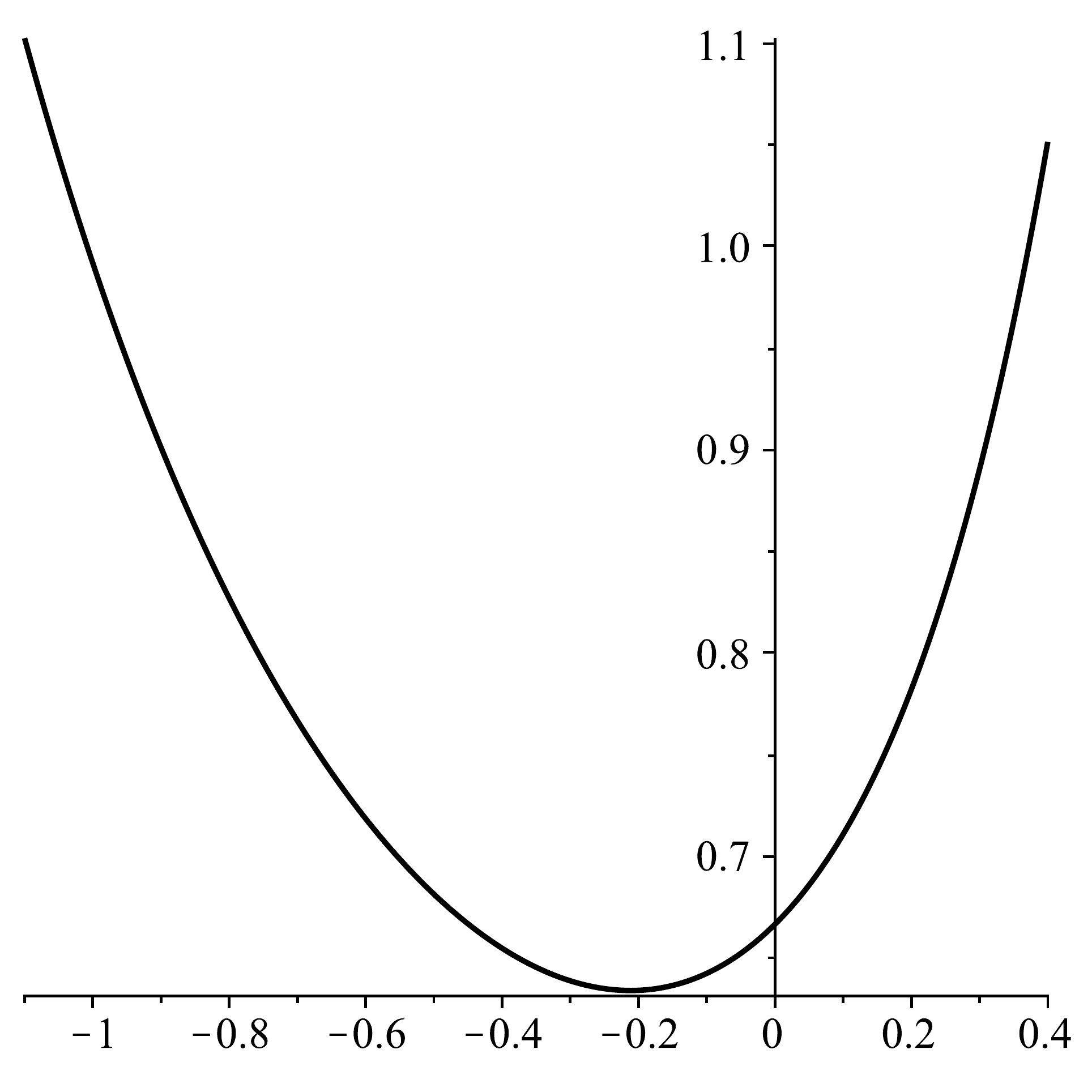}
\caption{
$Y(z)$ has singularities at $z_r\doteq0.3720$ and 
$z'\doteq-1.0079$, where 
$e^{(1+z)^2}(2/3e-\int_{1}^{1+z}e^{-u^2}du)=1$. 
}
\label{F_Y}
\end{figure}

\subsubsection{$N^v$ for $\A_f$}

In the case of $\A_f$ we have by \eqref{E_Xf}
that 
\[
Y(z)=\frac{1-z}
{1-2z(1+z)}.
\]
Therefore
\[
\E[T_f(\pi)^2]
=[z^n]Y(z)\sim 
z_f^{-n}/2
\]
where $z_f\doteq0.366$ 
is the smallest singularity of $Y$, and so
\[
N^v_f\sim 
\frac{5}{2\varphi^2} (z_f\varphi^2)^{-n}.
\]

\subsubsection{$N^v$ for $\A_g$}

First, we derive the generating function $X(z,t)$ for the sequence
$x_n=F_{n,1}\E_u[T(\pi)^t]$. 
The argument is similar to the proof of \eqref{E_Xf} given in \cref{L_momAf}, 
except that in this case we have 
$x_0=x_1=1$,  $x_2=2^{t+1}$
and for $n\ge3$, 
\[
x_n
=3^t(2x_{n-2}+x_{n-3}). 
\]
Hence
\[
X(z,t)
=
\frac{1+z+2z^2(2^t-3^t)}
{1-3^tz^2(2+z)}. 
\]
Therefore, setting $t=1$, 
\[
Y(z)=
\frac{1+z(1-2z)}
{1-3z^2(2+z)},
\]
and so 
\[
\E[T_g(\pi)^2]
=[z^n]Y(z)\sim 
\frac{1}{69}(8+17z_3+9z_3^2)
z_3^{-(n+1)}
\]
where $z_3\doteq0.3747$ 
is the smallest singularity of $Y$.
Hence
\[
N^v_g\sim 
\frac{5}{69}(8+17z_3+9z_3^2) (z_3\varphi^2)^{-(n+1)}.
\]

\subsection{Comparison of algorithms}
\label{S_compare}

The table appearing below compares the performance of the algorithms
studied in this section. 
The numbers $N^*$ are based on 
$e^{L+\sigma}$ in \eqref{E_e^L}. Note that 
the $e^{L+t}$ bounds of \cref{T_CD} are accurate if 
$\P(|\log Y-L|\ge t/2)$ is small. Here $\E \log Y=L$. 
More quantitative bounds follow by the one-sided Chebyshev 
inequalities 
\[
\P(\log Y-L\ge a\sigma)\le1/(1+a^2),\quad 
\P(\log Y\le  L-a\sigma)\le 1/(1+a^2).\]
Since we have proven normal approximation, we may 
replace $1/(1+a^2)$ with $\exp(-a^2/2)/(a\sqrt{2\pi})$. We have been 
content to use $e^{L+\sigma}$ in our numerical examples
since simulations show that this is a good indication of the 
sample size needed. 

On the other hand, the numbers $N^v$ are based on standard
relative variance considerations \eqref{E_relvar}. 
For small values of $n$, these are roughly comparable to 
the Kullback--Leibler numbers $N^*$; however, larger values of $n$ 
suggest that smaller
sample sizes than those given by $N^v$ are adequate.  

\begin{table}[h!]
\begin{center}
\caption{Comparison of algorithms
$\A_r$, $\A_f$ and $\A_g$
which match in random, fixed and greedy
orders. The values $N^*$ 
and $N^v$ are the estimates for required sample
size given by the Kullback--Leibler \eqref{E_e^L}
and relative variance \eqref{E_relvar}
criteria.}
\begin{tabular}{l|l|l|l|l}
\toprule
$n$&	200&	300&	500&1000\\ \hline
$N^*_r$
&$194$
&$12	11$
&$38257$
& $1.25\times10^8$\\
$N^v_r$
&$121$
&$1702$
&$3.37\times10^5$
& $1.86\times10^{11}$\\ \hline
$N^*_f$
&$1520$
&$22479$
&$3.75\times10^6$
& $6.72\times10^{11}$\\
$N^v_f$
&$4884$
&$3.50\times10^5$
&$1.79\times10^9$
& $3.34\times10^{18}$\\ \hline
$N^*_g$
&$75$
&$321$
&$4889$
&$2.79\times10^6$\\
$N^v_g$
&$54$
&$368$
&$17102$
&$2.54\times10^{8}$ \\ \hline
$n^7$
&$1.28\times10^{16}$
&$2.19\times10^{17}$
&$7.82\times10^{18}$
&$1\times10^{21}$\\
$F_{n,1}$
&$4.54\times10^{41}$
&$3.60\times10^{62}$
&$2.26\times10^{104}$
&$7.04\times10^{208}$\\
\bottomrule
\end{tabular}
\end{center}
\end{table}

\section{
2-Fibonacci matchings ${\mathcal F}_{n,2}$}\label{S_Fib2}

Next, we consider the class of 
2-Fibonacci matchings 
\[
{\mathcal F}_{n,2}=\{\pi\in S_n:-1\le \pi(i)-i\le 2\}.
\] 
In \cite{CDG18} 
$t$-Fibonacci matchings ${\mathcal F}_{n,t}$
are analyzed in general, however only for fixed order algorithms. 
Specifically, the mean and variance of $\log T(\pi)$ is computed
for algorithm $\A_{f,2}$ which matches in fixed order $1,2,\ldots,n$. 
For simplicity, we restrict ourselves to the case
$t=2$, although our methods extend to the general 
case. The cases $t>2$ present no additional challenges, 
apart from more complicated formulas. 
The case
$t=2$, on the other hand, is more involved than $t=1$, since 
unlike the case of Fibonacci matchings, 
after a single 
$\pi(i)$ is determined by an algorithm, the problem 
of determining $\pi$ 
does not necessarily split independently.
Note that, given that
$\pi(i)\in\{i,i+1,i+2\}$ for some  $1<i<n$, 
the values of $\pi$ 
on $[i-1]$ and $[n]-[i]$
are independent; {\it but} if $\pi(i)=i-1$, this is no longer the case since
$\pi(i-1)=i+1$ is possible 
in a 2-Fibonacci matching.  

\begin{figure}[h!]
\centering
\includegraphics[scale=1]{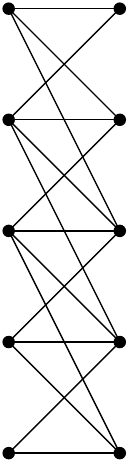}
\caption{Perfect matchings in this  graph  
correspond to $2$-Fibonacci permutations in ${\mathcal F}_{5,2}$. 
}
\end{figure} 

We consider sampling using   
algorithms $\A_{f,2}$, $\A_{r,2}$ and $\A_{g,2}$, which are analogues of 
$\A_{f}$, $\A_{r}$ and $\A_g$
studied in the previous section. Since the calculations become quite 
lengthy, we will highlight the main differences
with the Fibonacci case, but  only sketch the details.   

Algorithm
$\A_{f,2}$ matches in the fixed order
$1,2,\ldots,n$. This case is no more involved than that of 
$\A_{f}$
for Fibonacci matchings in any essential way, since this algorithm only ever assigns
some $\pi(i)=i-1$ when such a matching is forced due to previous matchings. 

On the other hand, some care is required to define tractable analogues  
$\A_{r,2}$ and $\A_{g,2}$.
Algorithm $\A_{r,2}$
matches in a random order, however when an index $i$ is selected, 
not only is $\pi(i)$ determined, but also the full cycle in $\pi$ containing $i$. 
In this way $\A_{r,2}$ matches
indices in a small (of minimal size) neighborhood of $i$
so that the further construction of $\pi$ splits into independent parts.
More specifically, in the first step of $\A_{r,2}$ a uniformly random $i\in[n]$
is selected. If $2<i<n-1$, there are six possibilities
(see \cref{F_Ar2}) 
for the cycle in $\pi$ containing $i$, namely
\begin{itemize}[nosep]
\item $\pi(i-2,i-1,i)=i,i-2,i-1$
\item $\pi(i-1,i)=i,i-1$
\item $\pi(i-1,i,i+1)=i+1,i-1,i$
\item $\pi(i)=i$
\item $\pi(i,i+1)=i+1,i$ 
\item $\pi(i,i+1,i+2)=i+2,i,i+1$. 
\end{itemize}
If $i\in\{1,2,n-1,n\}$ then only a subset of these are possible. Algorithm $\A_{r,2}$
selects one of these possibilities
at random.  
The algorithm continues in this way until a 2-Fibonacci permuation
$\pi$ is constructed. 
We note that this algorithm is a natural extension of the Fibonacci case 
$t=1$ case, since in that case choosing whether $\pi(i)$ is equal to $i-1$, $i$ or $i+1$ is equivalent 
to selecting the cycle containing $i$.

\begin{figure}[h!]
\centering
\includegraphics[scale=1]{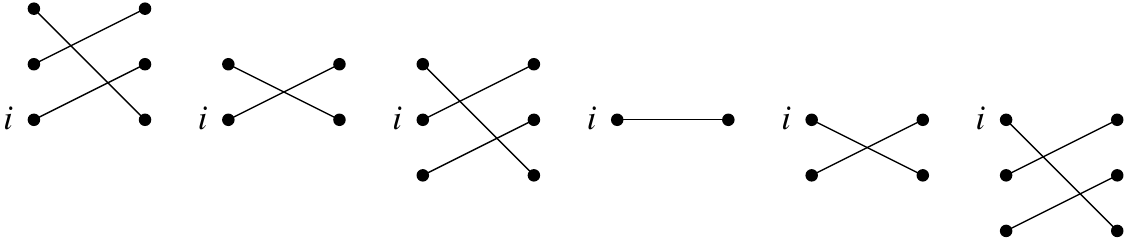}
\caption{If index $i$ is selected in some step of algorithm $\A_{r,2}$, 
one of the above configurations (or one amongst the remaining
allowable options) is selected uniformly at random for $\pi$ near $i$. 
Note that determining $\pi$ ``above and below'' this configuration
splits independently, since there are no $\pi(j)<j-1$ in a 
2-Fibonacci matching. 
}
\label{F_Ar2}
\end{figure} 

Finally, we consider the greedy algorithm $\A_{g,2}$ which in each step, starting with 
index 3,  
selects the minimal unmatched index $i$ amongst those with the 
maximal number of remaining 
allowable choices for the cycle containing $i$. 
Similarly as the $\A_{r,2}$ case, indices are matched in a neighborhood of $i$
by selecting a uniformly random cycle structure for unmatched indices $j\le i$, 
and we continue in this way until $\pi$ is constructed.

By the main results of this section, \cref{T_Af2,T_Ar2,T_Ag2}, and 
\cref{T_CD} we obtain estimates for the required sample size
$N^*$ listed in the table below. The greedy algorithm $\A_{g,2}$, for instance, 
outperforms $n^7$ until about $n=1549$. 

Approximate values of $F_{n,2}$ are obtained as follows. 
By considering whether $\pi(1)$
is equal to 1, 2 or 3, it follows that 
$F_{n,2}=F_{n-1,2}+F_{n-2,2}+F_{n-3,2}$
for $n\ge3$. Therefore
\begin{equation}\label{E_asy2fib}
F_{n,2}\sim c_2 \varphi_2^{n+1}
\end{equation}
where $\varphi_2\doteq1.8393$ satisfies
$\varphi_2^3=\varphi_2^2+\varphi_2+1$ and 
\[
c_2=\frac{1}{22}(3+\frac{7}{\varphi_2}+\frac{2}{\varphi_2^2})
\doteq 0.3363. 
\]

\begin{table}[h!]
\begin{center}
\caption{Comparison of algorithms
$\A_{r,2}$, $\A_{f,2}$ and $\A_{g,2}$
which match in random, fixed and greedy
orders. The values $N^*$ 
are the estimates for required sample
size given by the Kullback--Leibler \eqref{E_e^L}
criterion.}
\begin{tabular}{l|l|l|l|l}
\toprule
$n$&	200&	300&	500&1000\\\hline
$N^*_{r,2}$
&$1.48\times10^{5}$
&$1.49\times10^{7}$
&$1.04\times10^{11}$
& $1.64\times10^{20}$\\ 
$N^*_{f,2}$
&$5.52\times10^{8}$
&$2.16\times10^{12}$
&$1.95\times10^{19}$
& $1.26\times10^{36}$\\ 
$N^*_{g,2}$
&$9057$
&$2.81\times10^{5}$
&$2.00\times10^{8}$
&$1.27\times10^{15}$\\ \hline  
$n^7$
&$1.28\times10^{16}$
&$2.19\times10^{17}$
&$7.82\times10^{18}$
&$1\times10^{21}$\\
$F_{n,2}$
&$5.26\times10^{52}$
&$1.54\times10^{79}$
&$1.31\times10^{132}$
&$2.76\times10^{264}$\\
\bottomrule
\end{tabular}
\end{center}
\end{table}

\subsection{Random order algorithm $\A_{r,2}$}

We begin with the random order algorithm
$\A_{r,2}$, since it is the most involved of the three. 
Our discussion of $\A_{f,2}$ and $\A_{g,2}$
will be less detailed.

Recall that $\A_{r,2}$ selects a random index $i$ and then selects 
a cycle containing $i$ randomly from amongst the allowable options. The algorithm continues in this
way until a 2-Fibonacci permutation $\pi$ is obtained. 

\begin{thm}\label{T_Ar2}
Let $P_{r,2}(\pi)$ be the probability of $\pi$
under the random order algorithm $\A_{r,2}$, where   
$\pi$ is uniformly random with respect to 
$u=1/F_{n,2}$ on 
2-Fibonacci permutations $\pi\in{\mathcal F}_{n,2}$. 
Put $T_{r,2}(\pi)=P_{r,2}(\pi)^{-1}$.
As $n\to\infty$,  
\[
\frac{\log T_{r,2}(\pi)-\mu_{r,2} n}
{\sigma_{r,2}\sqrt{n}}\xrightarrow{d} N(0,1) 
\]
where
$\mu_{r,2}\doteq 0.6465$ and 
 $\sigma^2_{r,2}\doteq 0.0799$. 
\end{thm}

Although it is possible by our arguments to obtain closed form expressions 
(involving the implicitly
defined quantity $\varphi_2$ in \eqref{E_asy2fib})
for the quantities 
$\mu_{r,2}$ and  $\sigma^2_{r,2}$, we omit these
 as
they are quite complex. 

\begin{proof}
The proof is similar to that of \cref{L_Ar}. 
We define $x_n$ analogously (replacing $F_{n,1}$ 
and $T_r$ with 
$F_{n,2}$ and $T_{r,2}$). 
Then we have that 
$x_0=x_1=1$, $x_2=2^{t+1}$,  $x_3=(4/3)[3^t(1+2^{t})+4^{t}]$, 
and for $n\ge4$,
\begin{multline*}
nx_n
=2(3^t-6^t)(x_{n-1}+x_{n-2}+x_{n-3})
+2(5^t-6^t)(2x_{n-2}+2x_{n-3}+x_{n-4})\\
+6^t(\sum_{i=0}^{n-1}x_ix_{n-i-1}
+2\sum_{i=0}^{n-2}x_ix_{n-i-2}
+3\sum_{i=0}^{n-3}x_ix_{n-i-3}).
\end{multline*}
Therefore, along the same lines as \eqref{E_Xr}, 
we find that 
\begin{multline*}
\frac{\partial}{\partial z}X
-1-2^{t+2}z-4[3^t(1+2^{t})+4^{t}]z^2\\
=2(3^t-6^t)[(1+z+z^2)X-(1+2z+2(1+2^{t})z^2)]\\
+2(5^t-6^t)[z(2+2z+z^2)X-2z(1+2z)]\\
+6^t[(1+2z+3z^2)X^2-(1+4z+8(1+2^{t-1})z^2)] 
\end{multline*}
where $X(z,t)$ is the generating function
for the sequence of $x_n$. 
Using the asymptotics  \eqref{E_asy2fib} for $F_{n,2}$ and the above equation in place of \eqref{E_Xr}, 
it follows by the 
the proof of \cref{L_Ar} that 
\[
\E_u\log T_{r,2}(\pi)=\mu_{r,2}n+O(1),\quad 
\var_u\log T_{r,2}(\pi)=\sigma^2_{r,2}n+O(1). 
\]
Indeed, the same strategy of proof works essentially line-for-line, however with
more complicated expressions. We omit the details.

To conclude, put $Y_n=\log T_{r,2}(\pi)$ and $F_{n,2}=0$ for $n<0$. 
Note 
\[
Y_n
\overset{d}{=}
 Y^{(1)}_{I_1^{(n)}}+Y^{(2)}_{I_2^{(n)}}+b_n, 
\]
where, for $i\in[n]$, 
\begin{multline*}
\P((I_1^{(n)},I_2^{(n)})=(i_1,i_2))\\
=
\frac{1}{nF_{n,2}}\begin{cases}
F_{i-3,2}F_{n-i,2}& 
(i_1,i_2)=(i-3,n-i)\\ 
F_{i-2,2}F_{n-i,2}& 
(i_1,i_2)=(i-2,n-i)\\ 
F_{i-2,2}F_{n-i-1,2}& 
(i_1,i_2)=(i-2,n-i-1)\\ 
F_{i-1,2}F_{n-i,2}& 
(i_1,i_2)=(i-1,n-i)\\ 
F_{i-1,2}F_{n-i-1,2}& 
(i_1,i_2)=(i-1,n-i-1)\\
F_{i-1,2}F_{n-i-2,2}& 
(i_1,i_2)=(i-1,n-i-2)
\end{cases}
\end{multline*}
and $b_n=O(1)$ is equal to $\log3$, $\log5$ or $\log6$ if $i\in\{1,n\}$, $i\in\{2,n-1\}$
or $3\le i\le n-2$. 
Since 
$I_1^{(n)}\overset{d}{=}U_n+O(1)$ and 
$I_2^{(n)}\overset{d}{=} n-U_n+O(1)$ for $U_n$ uniform on $[n]$, the theorem 
follows by \cref{L_NR}. 
\end{proof}

\subsection{Fixed order algorithm $\A_{f,2}$}

Next, we turn to the simpler case of $\A_{f,2}$
which samples from ${\mathcal F}_{n,2}$ by
matching indices in $[n]$ in the fixed 
order $1,2,\ldots,n$, in each step setting
$\pi(i)$ equal to one of the remaining 
allowable options
uniformly at random.  Note that if 
previously we have set
$\pi(i-2)=i$ or $\pi(i-1)\in\{i,i+1\}$, then $\pi(i)=i-1$ is forced. 
Otherwise, we set $\pi(i)$ to $i$, $i+1$ or $i+2$ uniformly at random. 

\begin{thm}\label{T_Af2}
Let $P_{f,2}(\pi)$ be the probability of $\pi$
under the fixed order algorithm $\A_{f,2}$, where   
$\pi$ is uniformly random with respect to 
$u=1/F_{n,2}$ on 
2-Fibonacci permutations $\pi\in{\mathcal F}_{n,2}$. 
Put $T_{f,2}(\pi)=P_{f,2}(\pi)^{-1}$.
As $n\to\infty$,  
\[
\frac{\log T_{f,2}(\pi)-\mu_{f,2} n}
{\sigma_{f,2}\sqrt{n}}\xrightarrow{d} N(0,1) 
\]
where
$\mu_{f,2}\doteq 0.6794$ and 
 $\sigma^2_{f,2}\doteq 0.1592$. 
\end{thm}

\begin{proof}
The argument is similar to the proof of \cref{T_Af} for $\A_f$. 
(Alternatively, the central limit theorem
for renewal processes can be used, 
see the proof for $\A_{g,2}$ below.)
In this case, we obtain the generating function
\[
X(z,t)=
\frac{
1+(1-3^t)z+2(2^t-3^t)z^2
}{1-3^tz(1+z+z^2)} 
\]
for the sequence of $x_n=F_{n,2}\E_u[T_{f,2}(\pi)^t]$.
From this, it follows (arguing as in the proof \cref{S_musigf}) that 
\[
\E_u \log T_{f,2}(\pi)
=
\mu_{f,2} n+O(1),\quad
\var_u \log T_{f,2}(\pi)
=
\sigma^2_{f,2} n+O(1).
\]
We omit the details. Hence, as in the proof of \cref{T_Af}, the result follows
by \cref{L_NR} by considering a modified algorithm $\A_{f,2}^\eps$. 
\end{proof}

\subsection{Greedy algorithm $\A_{g,2}$}

As discussed above, algorithm $\A_{g,2}$
maximizes the number of steps with 6 choices for 
the cycle containing $i$. For $n\le4$, for simplicity, suppose that 
$\A_{g,2}$ selects a uniformly random 
$\pi\in{\mathcal F}_{n,2}$. For $n\ge5$, starting with index $3$, 
the smallest unmatched index is selected with the maximal number
of choices for the cycle containing $i$. Then a uniformly random cycle structure is chosen 
for unmatched indices $j\le i$, and 
the algorithm continues in this way until $\pi$ is constructed.  
Note that, for $i\le n-2$ there are 9 such cycle structures (see \cref{F_Ar2}), since if $\pi(i)\ge i$ then 
there are 2 choices for the cycle containing $i-1$, either $\pi(i-2,i-1)=i-2,i-1$ or $\pi(i-2,i-1)=i-1,i-2$.  

\begin{thm}\label{T_Ag2}
Let $P_{g,2}(\pi)$ be the probability of $\pi$
under the algorithm $\A_{g,2}$, where   
$\pi$ is uniformly random with respect to 
$u=1/F_{n,2}$ on 
2-Fibonacci permutations $\pi\in{\mathcal F}_{n,2}$. 
Put $T_{g,2}(\pi)=P_{g,2}(\pi)^{-1}$.
As $n\to\infty$,  
\[
\frac{\log T_{g,2}(\pi)-\mu_{g,2} n}
{\sigma_{g,2}\sqrt{n}}\xrightarrow{d} N(0,1) 
\]
where
$\mu_{g,2}\doteq 0.6365$ and 
 $\sigma^2_{g,2}\doteq 0.0514$. 
\end{thm}

\begin{proof}
This result follows in a similar way as \cref{T_Ag}, where in this case we compare 
$\log T_{g,2}(\pi)$ with $N_n\log9$, where 
$(N_t,t\ge0)$ is the renewal process whose inter-arrival times $X_i$ are equal to 
3, 4 and 5 with probabilities $4/\varphi_2^3$, $3/\varphi_2^4$ and $2/\varphi_2^5$. 
Since 
\[
\frac{1}{\E(X_1)}\doteq0.2897,\quad
\frac{\var(X_1)}{\E(X_1)^3}\doteq0.0106,
\]
the result follows by the central limit theorem for renewal processes. 
\end{proof}

\subsection{Almost perfect variants}

As in \cref{S_Aopt} we can, for instance, modify $\A_{f,2}$ to obtain an ``almost perfect''
version $\A_{f,2}^*$ (such that $\var_u\log(T_{f,2}^*(\pi)/F_{n,2})=O(1)$ 
under $u=1/F_{n,2}$) by instead setting $\pi(i)$ equal to $i$, $i+1$ and $i+2$
with probabilities $1/\varphi_2$, $1/\varphi_2^2$ and $1/\varphi_2^3$
(unless $\pi(i)=i-1$ is forced), corresponding to the asymptotic proportions of 
$\pi$ with $\pi(1)$ equal to 1, 2 and 3. Presumably this reasoning 
generalizes to all other cases $t\ge3$.

\section{
Distance-2 matchings ${\mathcal D}_{n,2}$}\label{S_Dis2}

Finally, we briefly discuss the case of distance-2 matchings 
\[
{\mathcal D}_{n,2}=\{\pi\in S_n:|\pi(i)-i|\le 2\}.
\]
This class is significantly  more complex 
than the previous cases ${\mathcal F}_{n,t}$, 
as matchings in  ${\mathcal D}_{n,2}$ can have cycles of all sizes. For instance, such a permutation 
can even be 
unicyclic (e.g., 
$\pi(12345)=24153\in{\mathcal D}_{5,2}$).
As a result, it seems quite involved to obtain the precise asymptotics for 
the mean and variance of 
$\log T(\pi)$ under a random order algorithm $\B_{r,2}$ (which 
like $\A_{r,2}$ selects a random vertex $i$ and then matches
around $i$ until the problem of determining $\pi$ splits into independent problems), 
so we do not pursue this here. 
However, assuming that $\log T(\pi)$ has asymptotically linear mean and variance, 
\cref{L_NR} would readily yield 
a central limit theorem. 

In this section, we restrict ourselves to algorithm $\B_{f,2}$
which matches in the fixed order $1,2,\ldots,n$. 
We also  
indicate how to obtain an ``almost perfect'' version $\B_{f,2}^*$
in \cref{S_A2opt} below. 

Distance-2 matchings are studied in Plouffe's thesis \cite{P92}, where in 
particular the generating function for $D_{n,2}=|{\mathcal D}_{n,2}|$ is derived. 
For completeness, we include the following lemma, which 
gives the asymptotics of $D_{n,2}$ and the related number
$D_{n,2}'$ of $\pi\in {\mathcal D}_{n,2}$ with $\pi(1)=2$.
See \cref{A_Dn2} for a proof. 

\begin{figure}[h!]
\centering
\includegraphics[scale=1]{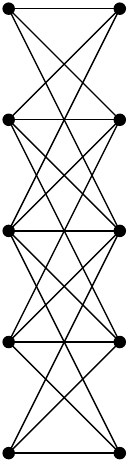}
\caption{Perfect matchings in this graph  
correspond to distance-2 permutations in ${\mathcal D}_{5,2}$. 
}
\end{figure}

\begin{lem}\label{L_D2}
For $n\ge4$, 
\[
D_{n+1,2}=2D_{n,2}+2D_{n-2,2}-D_{n-4,2} 
\]
and 
\[
2D_{n+2,2}'=D_{n+1,2}+D_{n,2}+D_{n-1,2}-D_{n-2,2}-D_{n-3,2}.
\]
\end{lem}

By \cref{L_D2}, we find that  
\begin{equation}\label{E_Dasy}
D_{n,2}\sim d_2\gamma_2^{n+1},\quad
D'_{n,2}/D_{n,2}\sim\gamma_2'
\end{equation}
where $\gamma_2\doteq 2.3335$ satisfies $\gamma_2^5=2\gamma_2^4+2\gamma_2^2-1$,
\[
d_2=
\frac{
4652\gamma_2^4
+10711\gamma_2^3
+3737\gamma_2^2
-3424\gamma_2
-2388
}{49163 \gamma_2^4}
\doteq 0.1948,
\]
and
\[\gamma_2'
=(\gamma_2^4+\gamma_2^3+\gamma_2^2-\gamma_2-1)/2\gamma_2^5
\doteq 0.3213.
\]

Note that, probabilistically, \eqref{E_Dasy} says that the chance that a random
$\pi\in {\mathcal D}_{n,2}$ has $\pi(1)=2$ is approximately $0.32$.

\subsection{Fixed order algorithm $\B_{f,2}$}

We consider the algorithm $\B_{f,2}$
which samples a distance-2 permutation
by matching the indices $[n]$ in the fixed order
$1,2,\ldots,n$, in each step matching 
$i$ with a uniformly random index
(amongst
the remaining allowable options). 

In establishing the 
following central limit theorem for 
$\log T_{f,2}(\pi)$, our arguments 
resemble those for \cref{T_Af,T_Af2}
(for $\A_f$ and $\A_{f,2}$), 
but are significantly 
complicated by the potential for long 
cycles in a distance-2 permutation. 
We sketch the main ideas. 

\begin{thm}\label{T_D2}  
Let $P_{f,2}(\pi)$ be the probability of $\pi$
under the fixed order algorithm $\B_{f,2}$, where   
$\pi$ is uniformly random with respect to 
$u=1/D_{n,2}$ on 
distance-2 permutations $\pi\in{\mathcal D}_{n,2}$. 
Put $T_{f,2}(\pi)=P_{f,2}(\pi)^{-1}$.
As $n\to\infty$,  
\[
\frac{\log T_{f,2}(\pi)-\mu_{f,2} n}
{\sigma_{f,2}\sqrt{n}}\xrightarrow{d} N(0,1) 
\]
where
$\mu_{f,2}\doteq 0.9053$ and
$\sigma_{f,2}^2\doteq0.1147$.
\end{thm}

Combining this result with \cref{T_CD}, we obtain the following 
estimates. 
Algorithm $\B_{f,2}$ outperforms $n^7$ until about $n=617$. 

\begin{table}[h!]
\label{T_d2}
\begin{center}
\caption{
Estimates for required sample $N^*_{f,2}$
size given by the Kullback--Leibler \eqref{E_e^L}
criterion for the fixed order algorithm $\B_{f,2}$
on distance-2 matchings ${\mathcal D}_{n,2}$. 
}
\begin{tabular}{l|l|l|l|l}
\toprule
$n$&	200&	300&	400&500\\\hline
$N^*_{f,2}$
&$2.85\times10^7$
&$2.74\times10^{10}$
&$2.23\times10^{13}$
& $1.63\times10^{16}$\\

$n^7$
&$1.28\times10^{16}$
&$2.19\times10^{17}$
&$1.64\times10^{18}$
&$7.81\times10^{18}$\\
$D_{n,2}$
&$1.82\times10^{73}$
&$1.15\times10^{110}$
&$7.26\times10^{146}$
&$4.59\times10^{183}$\\
\bottomrule
\end{tabular}
\end{center}
\end{table}

\begin{proof}
The mean and variance of 
$\log T_{f,2}(\pi)$ are derived in \cite{CDG18}
(by other methods), but for 
completeness, we present a concise derivation of these. 
As in the proofs in \cref{S_Fib}, 
put $x_n 
=D_{n,2} {\mathbb E}[T_{f,2}(\pi)^t]$ 
and $X(z,t)=\sum_{n=0}^\infty x_nz^n$. 
Recall that $D'_{n,2}$ is the number of 
$\pi\in{\mathcal D}_{n,2}$
with $\pi(1)=2$. 
Similarly, let $D''_{n,2}$ be the number of 
$\pi\in{\mathcal D}_{n,2}$
with $\pi(1)=3$. Define $x_n',X'$ and $x_n'',X''$ 
as for $x_n,X$. 

Observe that, for $n\ge3$, 
\[x_n=3^t(x_{n-1}+x'_n+x''_n)\]
since in the first step of $\B_{f,2}$, we set $\pi(1)$ equal to one of $1,2,3$
uniformly at random.
Likewise, for $n\ge4$, it is easy to see that 
\[
x'_n=3^t(x_{n-2}+x_{n-3}+x'_{n-2}),\quad
x''_n=3^t(x'_{n-1}+x_{n-3}+x_{n-4}).
\]
Hence 
\begin{align*}
&X-1-z-2^{t+1}z^2=3^t[z(X-1-z)+(X'-z^2)+X'']\\
&X'-z^2-2^{t+1}z^3=3^t[z^2(X-1-z)+z^3(X-1)+z^2X']\\
&X''-2^{t+1}z^3=3^t[z(X'-z^2)+z^3(X-1)+z^4X].
\end{align*}
Solving, we obtain 
$X(z,t)=U(z,t)/V(z,t)$, 
where
\begin{multline*}
U(z,t)=
1
-(3^t-1)z
+[2^{t+1}-(2+3^t)3^t]z^2\\
+3^t[2^{t+2}-1-(2+3^{t})3^t]z^3
+2(2^{t}-3^{t})3^t(3^t-1)z^4
-2(2^{t}-3^{t})3^{2t}z^5
\end{multline*}
and
\begin{multline*}
V(z,t)=
1-3^tz
-3^t(1+3^{t})z^2
-3^{2t}(1+3^{t})(1+z)z^3
+3^{3t}(1+z)z^5.
\end{multline*}
Using this generating function and the asymtotics \eqref{E_Dasy}, precise asymptotics 
for 
the mean and variance of $\log T_{f,2}(\pi)$ can be obtained,
arguing as in the proof of \cref{L_Ar}. 
We report only the numerical approximations 
\begin{equation}\label{E_Evar21}
\E_u\log T_{f,2}(\pi)= 0.9053 n +O(1),\quad 
\var_u\log T_{f,2}(\pi)= 0.1147 n +O(1). 
\end{equation}

Next, we apply \cref{L_NR} to obtain a central limit theorem for $\log T_{f,2}(\pi)$. 
As in the proofs of \cref{T_Af,T_Af2}, we consider a modified algorithm 
$\B_{f,2}^\eps$ and instead prove a central limit theorem for $\log T^\eps_{f,2}(\pi)$. 
Moreover, as in these proofs, we consider the matching of $k_n=\lfloor n/2\rfloor$
and try to split the problem of determining $\pi$ into two independent problems
``above and below'' $k_n$. However, due to the potential for long cycles in 
a $\pi\in {\mathcal D}_{n,2}$, the arguments are slightly more involved. 

We call a set of consecutive indices $I\subset[n]$ a {\it block} in $\pi$ 
if all indices in $I$ are matched
with indices in $I$. 
We view a permutation $\pi\in {\mathcal D}_{n,2}$ as a sequence of {\it blocks,}
beginning with the block containing index 1. 
As noted above, there are unicyclic $\pi\in {\mathcal D}_{n,2}$, and so possibly 
$[n]$ is the only block in $\pi$.
However, for $n>3$, there are  
only two ways in which this can occur (see \cref{F_nocycles}): 
\begin{itemize}[nosep]
\item
$\pi(1)=2$, $\pi(2)=4$ (so 
necessarily $\pi(3)=1$); and then $\pi(4)=6$
(so  $\pi(5)=3$); and then $\pi(6)=8$ (so  $\pi(7)=5$),  etc.\
until finally some $\pi(2i)=2i-1$.  
\item  
$\pi(1)=3$, $\pi(2)=1$; and then the above pattern is repeated
(but offset by one index) 
starting with $\pi(3)=5$ (so  $\pi(4)=2$), etc.
\end{itemize} 

\begin{figure}[h!]
\centering
\includegraphics[scale=1]{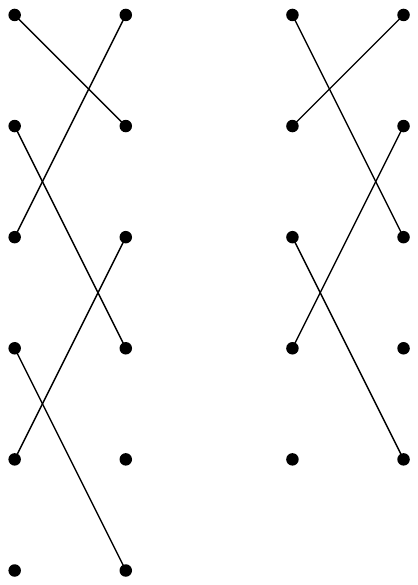}
\caption{Block-avoiding configurations in 
distance-2 matchings. 
}
\label{F_nocycles}
\end{figure} 

To complete the proof, let $K_n=\{x_n,\ldots,y_n\}$ 
be the block containing $k_n$ in a $\pi\sim u$. 
As noted already, $K_n$ can be large in general, 
and even possibly $K_n=[n]$. 
The key to applying \cref{L_NR}, 
however, is to note that since $\pi\sim u$, 
$K_n$ is very unlikely to be large. 
Indeed, since for $n>3$ there are only two block-avoiding patterns, 
it is easy to see that 
\begin{equation}\label{E_sizeKn}
\P_u(|K_n|=\ell)\le O(\ell \gamma_2^{-\ell})
\end{equation}
where $\gamma_2$ is as in \eqref{E_Dasy}. 
Thus, setting $Y_n=\log T_{f,2}^\eps(\pi)$, we have 
\[Y_n \overset{d}{=} 
Y^{(1)}_{x_n-1}+Y^{(2)}_{n-y_n}+b_n,
\]
where $b_n$ is the $\log$ of the probability that, 
given all indices in $[x_n-1]$ are matched correctly,  
$\B^\eps_{f,2}$ correctly matches those in $K_n$. 
By \eqref{E_sizeKn} the conditions of \cref{L_NR} are clearly satsified, 
giving the result. 
\end{proof}

\subsubsection{Almost perfect algorithm $\B_{f,2}^*$}
\label{S_A2opt}

Finally, we briefly explain how to modify 
algorithm $\B_{f,2}$ to obtain 
an ``almost perfect'' version $\B_{f,2}^*$. 
It is easiest to describe such 
an algorithm which in some steps determines the matching of more than 
one index. 
As in $\B_{f,2}$, indices are matched in order $1,2,\ldots,n$. 
If $i>n-4$, then we set $\pi(i)$ equal to one of the permissible options 
uniformly at random. If, on the other hand, $i\le n-4$, then 
we consider two cases. 
Recall $\gamma_2$ and $\gamma_2'$ in \eqref{E_Dasy}. 
 
\begin{enumerate} 
\item If all indices in $[i]$ have been matched with indices in $[i]$, then 
algorithm $\B_{f,2}^*$ sets
\begin{itemize}[nosep]
\item $\pi(i+1)=i+1$ with probability $1/\gamma_2$ 
\item $\pi(i+1)=i+2$ with probability $\gamma_2'$
\item $\pi(i+1,i+2)=i+3,i+1$ with probability $\gamma_2'/\gamma_2$
\item $\pi(i+1,i+2,i+3)=i+3,i+2,i+3$ with probability $1/\gamma_2^3$
\item $\pi(i+1,i+2,i+3,i+4)=i+3,i+4,i+1,i+2$ with probability
$1/\gamma_2^4$.   
\end{itemize}
\item On the other hand, if all indices in $[i]$ have been matched with indices in $[i+1]-\{i\}$,  then algorithm $\B_{f,2}^*$ sets
\begin{itemize}[nosep]
\item $\pi(i+1)=i$ with probability $1/\gamma_2'\gamma_2^2$ 
\item $\pi(i+1,i+2)=i+2,i$ with probability $1/\gamma_2'\gamma_2^3$
\item $\pi(i+1,i+2)=i+3,i$ with probability $1/\gamma_2^2$.
\end{itemize}
\end{enumerate}

The reason for the choices in (1) is that $1/\gamma_2$, $\gamma_2'$, 
$\gamma_2'/\gamma_2$,
$1/\gamma_2^3$ and $1/\gamma_2^4$ are the asymptotic proportions of matchings in 
${\mathcal D}_{n,2}$ with $\pi(1)=1$, $\pi(1)=2$, $\pi(12)=31$,
$\pi(123)=321$ and 
$\pi(1234)=3412$. 
The probabilities in (2) are chosen similarly, considering the asymptotic proportions of 
matchings in ${\mathcal D}_{n,2}'$ with $\pi(2)=1$, $\pi(23)=31$ and 
$\pi(23)=41$. 

Defining $x_n,X$ and $x_n',X'$ as in previous 
proofs, in this case we obtain 
$x_0=x_1=1$, $x_2=2^{t+1}$, 
$x_0'=x_1'=0$, $x_2'=1$
and $x_3'=2^{t+1}$, and for $n\ge4$, 
\begin{align*}
x_n&=
\gamma_2^tx_{n-1}
+(1/\gamma_2')^t x'_n
+(\gamma_2/\gamma_2')^t x'_{n-1}
+\gamma_2^{3t}x_{n-3}
+\gamma_2^{4t}x_{n-4}\\
x'_n&=
(\gamma_2'\gamma_2^2)^tx_{n-2}
+(\gamma_2'\gamma_2^3)^tx_{n-3}
+\gamma_2^{2t} x'_{n-2}. 
\end{align*}
Therefore
\begin{multline*}
X-1-z-2^{t+1}z^2-2^{t+1}3^{t+1}z^3\\
=
\gamma_2^tz(X-1-z-2^{t+1}z^2)
+(1/\gamma_2')^t(X'-z^2-2^{t+1}z^3)\\
+(\gamma_2/\gamma_2')^tz(X'-z^2)
+\gamma_2^{3t}z^3(X-1)
+\gamma_2^{4t}z^4X
\end{multline*}
and 
\[
X'-z^2-2^{t+1}z^3
=
(\gamma_2'\gamma_2^2)^tz^2(X-1-z)
+(\gamma_2'\gamma_2^3)^tz^3(X-1)
+\gamma_2^{2t}z^2 X'. 
\]
Hence 
$X(z,t)=U(z,t)/V(z,t)$, 
where
\begin{multline*}
U(z,t)=
1+(1-\gamma_2^t)z\\
+(2^{t+1}-\gamma_2^t-2\gamma_2^{2t})z^2
+(6^{t+1}-2(2\gamma_2)^t-2\gamma_2^{2t}-2\gamma_2^{3t})z^3\\
+\gamma_2^t(2(2/\gamma_2')^t+(\gamma_2/\gamma_2')^t
-2(2\gamma_2)^t
-\gamma_2^{3t})z^4\\
+\gamma_2^{2t}(-6^{t+1}+2(2/\gamma_2')^t
+2(2\gamma_2)^t+(\gamma_2/\gamma_2')^t
+\gamma_2^{3t})z^5
\end{multline*}
and
\[
V(z,t)=1-\gamma_2^tz-2\gamma_2^{2t}z^2
-2\gamma_2^{3t}z^3-2\gamma_2^{4t}z^4
+\gamma_2^{5t}z^5+\gamma_2^{6t}z^6.
\]
Using this generating function, it can be shown that 
\[
\E_u\log T(\pi)/n\to \log\gamma_2,
\quad 
\var_u\log T(\pi)=O(1). 
\]
Hence, by \cref{T_CD}, $N^*(n)=O(1)$ samples
are sufficient to well approximate ${\mathcal D}_{n,2}$ by this algorithm.

\appendix

\section{Supplementary results}

\subsection{Mean and variance for $\A_f$}\label{S_musigf}

In this section we establish \eqref{E_momAf}
giving the asymptotics for the mean and variance of $\log T_f(\pi)$
under the uniform distribution $u=1/F_{n,1}$. 

\begin{lem}\label{L_momAf}
As $n\to\infty$, 
\[
\E_u \log T_f(\pi)
\sim \mu_f n+
\frac{2(1-\sqrt{5})}{5}\log 2
\]
and
\[
\var_u \log T_f(\pi)\sim
\sigma^2_f n+ 
\frac{11\sqrt{5}-27}{25}\log^2 2, 
\]
where
\[
\mu_f=
\frac{1}{2}(1+\frac{1}{\sqrt{5}})\log 2,\quad 
\sigma^2_f
=\frac{1}{5\sqrt{5}}\log^2 2. 
\]
\end{lem}

\begin{proof}
As in the proof of \cref{L_Ar}, let $x_n=F_{n,1}\E_u[T_f(\pi)^t]$. 
Let $X(z,t)$ be the associated generating function, and note that 
\[
{\mathbb E}_u[(\log T_f(\pi))^k]
=
\frac{1}{F_{n,1}}
[z^n]\frac{\partial^k}{\partial t^k}X(z,0).
\]
Clearly, 
$x_0=x_1=1$. For $n\ge2$, note that $\P_u(\pi(1)=1)=F_{n-1,1}/F_{n,1}$
and $\P_u(\pi(1)=2)=F_{n-2,1}/F_{n,1}$, and in either case $\A_f$
correctly matches $1$ with $\pi(1)$ with probability $1/2$. Therefore, for $n\ge2$, we have that 
\[
x_n=2^t(x_{n-1}+x_{n-2}). 
\]
Summing over $n\ge2$, it follows that 
\[
X-1-z=2^t[z(X-1)+z^2X]
\]
and so, as stated in  \eqref{E_Xf}, 
\[
X(z,t)
=
\frac{1+(1-2^t)z}
{1-2^tz(1+z)}. 
\]
Therefore  
\[
(1-z-z^2)^2\frac{\partial}{\partial t} X(z,t)
=
z^2(2+z)\log 2
\]
and
\[
(1-z-z^2)^3\frac{\partial^2}{\partial t^2} X(z,t)
=
z^2(2+z)(1+z+z^2)\log^2 2. 
\]
Hence the lemma follows, arguing as in 
the proof of \cref{L_Ar}, using the asymptotics 
\eqref{E_Fib2} and \eqref{E_Fib3}. 
Indeed, let $A_n$ and $B_n$ denote the right hand sides
of \eqref{E_Fib2} and \eqref{E_Fib3}. 
By the above arguments, we find that 
\[
\E_u\log T_f(\pi)\sim 
[\frac{A_{n-3}}{\varphi^3}+2\frac{A_{n-2}}{\varphi^2}]\log2\\
= \mu_f n+\frac{2(1-\sqrt{5})}{5}\log 2
\]
and 
\begin{multline*}
\E_u[(\log T_f(\pi))^2]\sim 
[\frac{B_{n-5}}{\varphi^5}
+3\frac{B_{n-4}}{\varphi^4}
+3\frac{B_{n-3}}{\varphi^3}
+2\frac{B_{n-2}}{\varphi^2}
]\log^2 2\\
=(\mu_f n)^2-\frac{1}{25}[7\sqrt{5}n+3(\sqrt{5}-1)]\log^2 2. 
\end{multline*}
Hence 
\[
\var_u\log T_f(\pi)\sim \sigma_f^2
+\frac{11\sqrt{5}-27}{25}\log^2 2
\]
completing the proof. 
\end{proof}

\subsection{Distance-2 asymptotics}\label{A_Dn2}

In order to prove \cref{L_D2}, we first establish the 
following observation.

\begin{lem}\label{L_D1}
For $n\ge4$, 
\[
D_{n,2}
=D_{n-1,2}+D'_{n,2}+D'_{n-1,2}
+D_{n-3,2}+D_{n-4,2}
\]
and 
\[
D'_{n,2}=D_{n-2,2}+D_{n-3,2}+D'_{n-2,2}
=\sum_{k=0}^{n-2} D_{k,2}.
\]
\end{lem}
\begin{proof}
Indeed, to see the first equality,  note that there are 
$D_{n-1,2}$ matchings with $\pi(1)=1$; $D'_{n,2}$ with $\pi(1)=2$;  
 $D'_{n-1,2}$ with $\pi(1)=3$
and $\pi(2)=1$; $D_{n-3,2}$ with $\pi(1)=3$ and $\pi(2)=2$ (so 
necessarily $\pi(3)=1$); and $D_{n-4,2}$ with $\pi(1)=3$ and $\pi(2)=4$
(so necessarily $\pi(3)=1$ and $\pi(4)=2$). 
The second equality is similarly derived. 
\end{proof}

Using these formulas, we deduce the result. 

\begin{proof}[Proof of \cref{L_D2}]
The first equality follows immediately by \cref{L_D1}, noting that 
for all $n\ge4$, 
\[
D_{n+1,2}-D_{n,2}
=D_{n,2}-D_{n-1,2}+D_{n-2,2}-D_{n-4,2}
+D'_{n+1,2}-D'_{n-1,2}
\]
and
$D'_{n+1,2}-D'_{n-1,2} = D_{n-1,2}+D_{n-2,2}$. 

Then, for the second equality, note that by the first, for all $k\ge4$, we have 
\[
D_{k+1,2}-D_{k,2}=D_{k,2}+2D_{k-2,2}-D_{k-4,2}. 
\]
Summing over $4\le k\le n$, we find 
\begin{multline*}
D_{n+1,2}-D_{4,2}=2\sum_{k=0}^{n}D_{k,2}
-(D_{0,2}+D_{1,2}+D_{2,2}+D_{3,2})\\
-2(D_{0,2}+D_{1,2}+D_{n-1,2}+D_{n,2})
+(D_{n-3,2}+D_{n-2,2}+D_{n-1,2}+D_{n,2}),
\end{multline*}
and the result follows noting that $D_{0,2}=D_{1,2}=1$, 
$D_{2,2}=2$, $D_{3,2}=6$ and $D_{4,2}=14$, and recalling that, by 
\cref{L_D1}, $D'_{n+2,2}=\sum_{k=0}^{n}D_{k,2}$. 
\end{proof}

\providecommand{\bysame}{\leavevmode\hbox to3em{\hrulefill}\thinspace}
\providecommand{\MR}{\relax\ifhmode\unskip\space\fi MR }
\providecommand{\MRhref}[2]{%
  \href{http://www.ams.org/mathscinet-getitem?mr=#1}{#2}
}
\providecommand{\href}[2]{#2}


\begin{thebibliography}{10}

\bibitem{B90}
R.~B. Bapat, \emph{Permanents in probability and statistics}, Linear Algebra
  Appl. \textbf{127} (1990), 3--25.

\bibitem{BHJ92}
A.~D. Barbour, L.~Holst, and S.~Janson, \emph{Poisson approximation}, Oxford
  Studies in Probability, vol.~2, The Clarendon Press, Oxford University Press,
  New York, 1992, Oxford Science Publications.

\bibitem{B16}
A.~Barvinok, \emph{Combinatorics and complexity of partition functions},
  Algorithms and Combinatorics, vol.~30, Springer, Cham, 2016.

\bibitem{BSVV08}
I.~Bez\'{a}kov\'{a}, D.~\v{S}tefankovi\v{c}, V.~V. Vazirani, and E.~Vigoda,
  \emph{Accelerating simulated annealing for the permanent and combinatorial
  counting problems}, SIAM J. Comput. \textbf{37} (2008), no.~5, 1429--1454.

\bibitem{BD10}
J.~Blitzstein and P.~Diaconis, \emph{A sequential importance sampling algorithm
  for generating random graphs with prescribed degrees}, Internet Math.
  \textbf{6} (2010), no.~4, 489--522.

\bibitem{B12}
O.~Blumberg, \emph{Cutoff for the transposition walk on permutations with
  one-sided restrictions}, preprint (2012), available at
  \href{https://arxiv.org/abs/1202.4797}{arXiv:1202.4797}.

\bibitem{Olena}
\bysame, \emph{Permutations with interval restrictions}, Ph.D. thesis, Stanford
  University, 2012.

\bibitem{B73}
L.~M. Br\`egman, \emph{Certain properties of nonnegative matrices and their
  permanents}, Dokl. Akad. Nauk SSSR \textbf{211} (1973), 27--30.

\bibitem{CD18}
S.~Chatterjee and P.~Diaconis, \emph{The sample size required in importance
  sampling}, Ann. Appl. Probab. \textbf{28} (2018), no.~2, 1099--1135.

\bibitem{CDHL05}
Y.~Chen, P.~Diaconis, S.~P. Holmes, and J.~S. Liu, \emph{Sequential {M}onte
  {C}arlo methods for statistical analysis of tables}, J. Amer. Statist. Assoc.
  \textbf{100} (2005), no.~469, 109--120.

\bibitem{CDG18}
F.~Chung, P.~Diaconis, and R.~Graham, \emph{Permanental generating functions
  and sequential importance sampling}, Adv.\ in Appl.\ Math., in press,
  available online (2019) at \href{https://doi.org/10.1016/j.aam.2019.05.004}
  {https://doi.org/10.1016/j.aam.2019.05.004}.

\bibitem{CDGM81}
F.~R.~K. Chung, P.~Diaconis, R.~L. Graham, and C.~L. Mallows, \emph{On the
  permanents of complements of the direct sum of identity matrices}, Adv. in
  Appl. Math. \textbf{2} (1981), no.~2, 121--137.

\bibitem{BKMR05}
P.-T. de~Boer, D.~P. Kroese, S.~Mannor, and R.~Y. Rubinstein, \emph{A tutorial
  on the cross-entropy method}, Ann. Oper. Res. \textbf{134} (2005), 19--67.

\bibitem{D18}
P.~Diaconis, \emph{Sequential importance sampling for estimating the number of
  perfect matchings in bipartite graphs: An ongoing conversation with laci},
  Building Bridges II, Bolyai Society Mathematical Studies, vol.~28, Springer,
  Berlin, Heidelberg, 2019, pp.~223--233.

\bibitem{DG81}
P.~Diaconis and R.~Graham, \emph{The analysis of sequential experiments with
  feedback to subjects}, Ann. Statist. \textbf{9} (1981), no.~1, 3--23.

\bibitem{DGH99}
P.~Diaconis, R.~Graham, and S.~P. Holmes, \emph{Statistical problems involving
  permutations with restricted positions}, State of the art in probability and
  statistics ({L}eiden, 1999), IMS Lecture Notes Monogr. Ser., vol.~36, Inst.
  Math. Statist., Beachwood, OH, 2001, pp.~195--222.

\bibitem{DJM}
M.~Dyer, M.~Jerrum, and H.~M\"{u}ller, \emph{On the switch {M}arkov chain for
  perfect matchings}, Proceedings of the {T}wenty-{S}eventh {A}nnual
  {ACM}-{SIAM} {S}ymposium on {D}iscrete {A}lgorithms, ACM, New York, 2016,
  pp.~1972--1983.

\bibitem{EP92}
B.~Efron and V.~Petrosian, \emph{A simple test of independence for truncated
  data with applications to red shift surveys}, Astrophysical Journal
  \textbf{399} (1992), 345--352.

\bibitem{FS09}
P.~Flajolet and R.~Sedgewick, \emph{Analytic combinatorics}, Cambridge
  University Press, Cambridge, 2009.

\bibitem{HN02}
H.-K. Hwang and R.~Neininger, \emph{Phase change of limit laws in the quicksort
  recurrence under varying toll functions}, SIAM J. Comput. \textbf{31} (2002),
  no.~6, 1687--1722.

\bibitem{MSV04}
M.~Jerrum, A.~Sinclair, and E.~Vigoda, \emph{A polynomial-time approximation
  algorithm for the permanent of a matrix with nonnegative entries}, J. ACM
  \textbf{51} (2004), no.~4, 671--697.

\bibitem{DKpc}
D.~Knuth, private communication.

\bibitem{Liu08}
J.~S. Liu, \emph{Monte {C}arlo strategies in scientific computing}, Springer
  Series in Statistics, Springer, New York, 2008.

\bibitem{LP86}
L.~Lov\'{a}sz and M.~D. Plummer, \emph{Matching theory}, North-Holland
  Mathematics Studies, vol. 121, North-Holland Publishing Co., Amsterdam;
  North-Holland Publishing Co., Amsterdam, 1986, Annals of Discrete
  Mathematics, 29.

\bibitem{Melczer}
S.~Melczer, \emph{An invitation to analytic combinatorics: From one to several
  variables}, Springer Texts \& Monographs in Symbolic Computation (undergoing
  editing), 2020.

\bibitem{MR92}
P.~R. Milgrom and J~Roberts, \emph{Economics, organization, and management},
  Prentice Hall, 1992.

\bibitem{NRpc}
R.~Neininger, private communication.

\bibitem{N15}
\bysame, \emph{Refined {Q}uicksort asymptotics}, Random Structures Algorithms
  \textbf{46} (2015), no.~2, 346--361.

\bibitem{NR04}
R.~Neininger and L.~R\"{u}schendorf, \emph{A general limit theorem for
  recursive algorithms and combinatorial structures}, Ann. Appl. Probab.
  \textbf{14} (2004), no.~1, 378--418.

\bibitem{P99}
B.~Pittel, \emph{Normal convergence problem? {T}wo moments and a recurrence may
  be the clues}, Ann. Appl. Probab. \textbf{9} (1999), no.~4, 1260--1302.

\bibitem{P92}
S.~Plouffe, \emph{Approximations de s\'eries g\'en\'eratrices et quelques
  conjectures}, Ph.D. thesis, Universit\'e du Qu\'ebec \`a Montr\'eal, 1992.

\bibitem{Andy}
A.~Tsao, \emph{Sampling algorithms for intractable counting problems}, Ph.D.
  thesis, Stanford University, 2020.

\bibitem{V79}
L.~G. Valiant, \emph{The complexity of computing the permanent}, Theoret.
  Comput. Sci. \textbf{8} (1979), no.~2, 189--201.

\end{thebibliography}
\end{document}